\documentclass[10pt]{article}

 \usepackage{color}

\usepackage{amsmath,amscd,amsfonts,amsthm, xypic, amssymb}
\input{xy}

\newcommand{\shrinkmargins}[1]{
  \addtolength{\textheight}{#1\topmargin}
  \addtolength{\textheight}{#1\topmargin}
  \addtolength{\textwidth}{#1\oddsidemargin}
  \addtolength{\textwidth}{#1\evensidemargin}
  \addtolength{\topmargin}{-#1\topmargin}
  \addtolength{\oddsidemargin}{-#1\oddsidemargin}
  \addtolength{\evensidemargin}{-#1\evensidemargin}
  }

\shrinkmargins{.7}

\DeclareMathOperator{\Hom}{Hom}

\DeclareMathOperator{\Isom}{Isom}
\DeclareMathOperator{\detpm}{det_{\pm}}
\DeclareMathOperator{\tdet}{\overline{\detpm}}

\DeclareMathOperator{\GL}{GL}

\DeclareMathOperator{\End}{End}
\DeclareMathOperator{\eend}{end}

\DeclareMathOperator{\Spec}{Spec}
\DeclareMathOperator{\Spf}{Spf}

\DeclareMathOperator{\Pic}{Pic}

\DeclareMathOperator{\Gal}{Gal}

\DeclareMathOperator{\Aut}{Aut}

\DeclareMathOperator{\proj}{proj}
\DeclareMathOperator{\gl}{gl}
\DeclareMathOperator{\pic}{pic}

\DeclareMathOperator{\loc}{loc}

\DeclareMathOperator{\Tot}{Tot}
\DeclareMathOperator{\Cotor}{Cotor}
\DeclareMathOperator{\Map}{Map}

\DeclareMathOperator*{\colim}{colim}

\DeclareMathOperator*{\hofib}{hofib}
\DeclareMathOperator*{\cofib}{cofib}
\DeclareMathOperator*{\hoeq}{hoeq}

\newcommand{\Sdet}{S \langle \det \rangle}
\newcommand{\Sdetold}{S \langle \det \rangle}

\newcommand{\field}[1]{\mathbb{#1}}

\newcommand{\Q}{\field{Q}}

\newcommand{\Z}{\field{Z}}
\newcommand{\N}{\field{N}}
\newcommand{\F}{\field{F}}
\newcommand{\W}{\field{W}}

\newcommand{\R}{\field{R}}
\newcommand{\C}{\field{C}}
\newcommand{\G}{\field{G}}
\newcommand{\SGn}{S\G_n^{\pm}}
\newcommand{\bS}{\field{S}}

\newcommand{\OO}{\mathcal{O}}

\newcommand{\id}{\mbox{id}}
\newcommand{\im}{{\rm im \,}}

\newcommand{\m}{\mathfrak{m}}

\newcommand{\ubar}{\overline{u}}

\newcommand{\mup}{\mu_{p^\infty}}

\newcommand{\beq}{\begin{displaymath}}
\newcommand{\eeq}{\end{displaymath}}
\newcommand{\beqn}{\begin{equation}}
\newcommand{\eeqn}{\end{equation}}

\theoremstyle{plain}
\newtheorem{thm}{Theorem}[section]
\newtheorem{prop}[thm]{Proposition}
\newtheorem{cor}[thm]{Corollary}
\newtheorem{lem}[thm]{Lemma}

\theoremstyle{definition}
\newtheorem{defn}[thm]{Definition}

\newtheorem{exmp}[thm]{Example}
\newtheorem{exmps}[thm]{Examples}

\theoremstyle{remark}
\newtheorem{rem}[thm]{Remark}

\title{A higher chromatic analogue of the image of J}
\author{Craig Westerland}

\begin{document}

\bibliographystyle{amsalpha}

\maketitle

\begin{abstract}

We prove a higher chromatic analogue of Snaith's theorem which identifies the K-theory spectrum as the localisation of the suspension spectrum of $\C P^\infty$ away from the Bott class; in this result, higher Eilenberg-MacLane spaces play the role of $\C P^\infty = K(\Z, 2)$.  Using this, we obtain a partial computation of the part of the Picard-graded homotopy of the $K(n)$-local sphere indexed by powers of a spectrum which for large primes is a shift of the Gross-Hopkins dual of the sphere.  Our main technical tool is a $K(n)$-local notion generalising complex orientation to higher Eilenberg-MacLane spaces.  As for complex-oriented theories, such an orientation produces a one-dimensional formal group law as an invariant of the cohomology theory.  As an application, we prove a theorem that gives evidence for the chromatic redshift conjecture.

\end{abstract}

\section{Introduction}

The stable homotopy groups of a space $X$ are defined as the colimit 
$$\pi_j^S(X) = \lim_{m \to \infty} \pi_{j+m} (\Sigma^m X) = \lim_{m \to \infty} \pi_j (\Omega^m \Sigma^m X) = \pi_j(QX).$$
where $QX = \varinjlim \Omega^m \Sigma^m X$.  

The J-homomorphism $J: \pi_j(O) \to \pi_j^S(S^0)$ may be regarded as a first approximation to the stable homotopy groups of $S^0$; here $O$ denotes the infinite orthogonal group $\varinjlim O(m)$.  It is induced in homotopy by the limit over $m$ of maps
$$J_m: O(m) \to \Omega^m S^m,$$
where for a matrix $M \in O(m)$ regarded as a linear transformation $M: \R^m \to \R^m$, $J_m(M) = M\cup \{\infty\}: S^m \to S^m$.  There is an analogous function from the infinite unitary group $U$; $J_U: U \to QS^0$ is given by composition with the forgetful map $U \to O$.  The homotopy groups of the domains are computable via Bott periodicity, and are 
$$\begin{array}{|c|c|c|c|c|c|c|c|c|}
\hline
j \bmod 8 & 0 & 1 & 2 & 3 & 4 & 5 & 6 & 7 \\
\hline
\pi_j (O) & \Z / 2 & \Z / 2 & 0 & \Z & 0 & 0 & 0 & \Z \\
\hline
\pi_j (U) & 0 & \Z & 0 & \Z & 0 & \Z & 0 & \Z \\
\hline
\end{array}$$

The work of Adams \cite{adams} shows that for an odd prime $p$, in dimensions $3 \bmod 4$, the $p$-torsion of the image in $\pi_j^S(S^0)$ of the cyclic group $\pi_j(O)$ is isomorphic to $\Z/ p^{k+1}$ when we can write $j+1 = 2(p-1)p^k m$ with $m$ coprime to $p$.  When $j$ cannot be written in this form, the $p$-torsion in the image of $J$ is zero.  If we are working away from $p=2$, this computation may be done using $U$ in place of $O$.  A computation of the 2-torsion in the image of $J$ follows from Quillen's proof of the Adams conjecture \cite{quillen}.

As this result indicates, the burden of computation in stable homotopy theory often encourages one to work locally at a prime $p$.  That is, one computes the $p$-power torsion in $\pi_*^S(X)$, and then assembles the results together into an integral statement.  One of the deeper insights of homotopy theorists in the last half century is that one may, following Bousfield \cite{bousfield}, do such computations local to any \emph{cohomology theory} $E^*$ to get more refined computations.  Morava's extraordinary K-theories $K(n)$ and E-theories $E_n$ \cite{morava} have proven particularly suited to this purpose; see, e.g., \cite{mrw, hopkins_smith}.

When $n=1$, $K(1)$ is identified with (a split summand of) mod $p$ K-theory.  The fact that $\pi_*(U) = \pi_{*+1}(K)$ for $*>0$ suggests that Adams' computation of the $p$-torsion in the image of $J$ is related to $K(1)$-local homotopy theory.  This is in fact the case; the localisation map
$$\pi_*^S(S^0) \to \pi_*(L_{K(1)} S^0)$$
carries $\im J$ isomorphically onto the codomain in positive degrees.  

One substantial difference between the stable homotopy category and its $K(n)$-local variant is the existence of exotic invertible elements.  In the stable homotopy category, the only spectra which admit inverses with respect to the smash product are spheres; thus the \emph{Picard group} of equivalence classes of such spectra is isomorphic to $\Z$.  In contrast, the Picard group of the $K(n)$-local category, $\Pic_n$, includes $p$-complete factors as well as torsion (see, e.g., \cite{hms, ghmr_pic}).

Our main result is a computation of part of the \emph{Picard graded} homotopy of the $K(n)$-local sphere analogous to the image of $J$ computation.  Throughout this paper $p$ will denote an odd prime; when localising with respect to $K(n)$, the prime $p$ is implicitly used.  We will use the notation $S: = L_{K(n)} S^0$ for the $K(n)$-local sphere and $A \otimes B := L_{K(n)} (A \wedge B)$ for the $K(n)$-local smash product.

\begin{thm} \label{main_intro_thm}

Let $\ell \in \Z$, and write $\ell=p^k m$, where $m$ is coprime to $p$.  Then the group $[\Sdet^{\otimes \ell(p-1)}, L_{K(n)} S^1]$ contains a subgroup isomorphic to $\Z / p^{k+1}$.  Furthermore, if $n^2<2p-3$, there is an exact sequence
$$0 \to \Z / p^{k+1} \to [\Sdet^{\otimes \ell(p-1)}, L_{K(n)} S^1] \to N_{k+1} \to 0$$
where $N_{k+1} \leq \pi_{-1}(S)$ is the subgroup of $p^{k+1}$-torsion elements.

\end{thm}

This is proved as Corollary \ref{final_cor}.  Here, $\Sdet \in \Pic_n$ was introduced by Goerss et al. in \cite{ghmr_pic}; it is defined below.  When $n=1$ and $p>2$, $\Sdet$ may be identified as $L_{K(1)}S^2$, and so this result recovers the constructive part of the classical image of $J$ computation.  More generally, $\Sdet$ may be identified as a shift of the Brown-Comenetz dual of the $n^{\rm th}$ monochromatic layer of the sphere spectrum if $\max\{2n+2, n^2 \} < 2(p-1)$ (see \cite{gross-hopkins}).  This identification fails for small primes; see, e.g. \cite{goerss_henn} for the case $n=2$ and $p=3$.

\subsection{The invertible spectrum $\Sdet$}

Morava's $E$-theories are Landweber exact cohomology theories $E_n$ associated to the universal deformation of the Honda formal group $\Gamma_n$ from $\F_{p^n}$ to $\W(\F_{p^n})[[u_1, \dots, u_{n-1}]][u^{\pm 1}]$.  When $n=1$, $E_1$ is precisely $p$-adic $K$-theory.  The Goerss-Hopkins-Miller theorem \cite{gh, gh2} equips the spectrum $E_n$ with an action of the \emph{Morava stabiliser group} 
$$\G_n := \Gal(\F_{p^n}/\F_p) \ltimes \Aut(\Gamma_n)$$
which lifts the defining action in homotopy.  The work of Devinatz-Hopkins \cite{dev_hop}, Davis \cite{davis, davis_iterated}, and Behrens-Davis \cite{bd} then allows one to define continuous\footnote{All homotopy fixed point spectra considered in this article will be of the continuous sort.} homotopy fixed point spectra with respect to closed subgroups of $\G_n$ in a consistent way.  The homotopy fixed point spectrum of the full group is the $K(n)$-local sphere: $E_n^{h\G_n} \simeq L_{K(n)} S^0$.  

The automorphism group $\Aut(\Gamma_n)$ is known to be the group of units of an order of a rank $n^2$ division algebra over $\Q_p$; the determinant of the action by  left multiplication defines a homomorphism $\det: \Aut(\Gamma_n) \to \Z_p^{\times}$.   Extend this to a homomorphism 
$$\detpm: \G_n \to \Z_p^{\times}$$
by sending the Frobenius generator of $\Gal(\F_{p^n}/\F_p)$ to $(-1)^{n-1}$.  We will write $\SGn$ for the kernel of this map.  We may define the homotopy fixed point spectrum $E_n^{h\SGn}$ for the restricted action of this subgroup.  We note: in many references, such as \cite{gross-hopkins, ghmr_pic}, a different extension $\det:\G_n \to \Z_p^{\times}$ is used; there $\Gal(\F_{p^n}/\F_p)$ is in the kernel of $\det$.  In an earlier version of this article, we incorrectly used this older form of $\det$; we thank Charles Rezk for clarifying this point.  Nonetheless, the homotopy fixed point spectra for the kernels of $\det$ and $\detpm$ are equivalent (though not as ring spectra); see section \ref{old_R_n_section}.

The spectrum $E_n^{h\SGn}$ retains a residual action of $\Z_p^{\times} = \G_n / \SGn$; for an element $k \in \Z_p^{\times}$, we will write the associated map as 
$$\psi^k: E_n^{h\SGn} \to E_n^{h\SGn}.$$
The reader is encouraged to think of these automorphisms as analogues of Adams operations.  Noting that $\Z_p^{\times} = \mu_{p-1} \times (1+p\Z_p)^{\times}$ is topologically cyclic with generator $g = \zeta_{p-1}(1+p)$, we define $F_\gamma$ as the homotopy fibre of 
$$\psi^g-\gamma: E_n^{h\SGn} \to E_n^{h\SGn}$$
for any $\gamma \in \Z_p^{\times}$.  These spectra are always invertible, and in fact the construction $\gamma \mapsto F_\gamma$ defines a homomorphism $\Z_p^{\times} \to \Pic_n$.  When $\gamma =1$, the associated homotopy fibre defines the homotopy fixed point spectrum for the action of $\Z_p^{\times}$, and so
$$F_1 = (E_n^{h\SGn})^{h\Z_p^{\times}} \simeq E_n^{h\G_n} \simeq L_{K(n)} S^0$$
In contrast, one defines $\Sdet := F_g$. 

\subsection{A Snaith theorem}

The K-theory spectrum admits a remarkable description due to Snaith \cite{snaith}.  He shows that the natural inclusion of $\C P^\infty$ into $BU \times \Z$ localises to an equivalence
$$\Sigma^\infty \C P^\infty_+ [\beta^{-1}] \simeq K.$$
Here the \emph{Bott map} $\beta: S^2 \to \Sigma^\infty \C P^\infty_+$ is a reduced, stable form of the inclusion $\C P^1 \to \C P^\infty$.

\begin{thm} \label{intro_snaith_thm}

There is a map $\rho_n: \Sdet \to L_{K(n)} \Sigma^\infty K(\Z_p, n+1)_+$ and an equivalence of $E_\infty$-ring spectra
$$L_{K(n)} \Sigma^\infty K(\Z_p, n+1)_+[\rho_n^{-1}] \to E_n^{h\SGn}.$$

\end{thm}

We will refer to the equivalent ring spectra of the theorem as $R_n$.  One inverts the map $\rho_n$ in the same fashion as for normal homotopy elements of a ring spectrum, via a telescope construction.  A consequence of this theorem is that the homotopy of $R_n$ is $\Sdet$-periodic.  When $n=1$, $R_1$ is simply $p$-adic K-theory, and this result is familiar as Bott periodicity.  Indeed, the map $\rho_1$ is the $K(1)$-localisation of $\beta$, and so we refer to $\rho_n$ as a \emph{higher Bott map}.  

The proof of this theorem uses the computations of Ravenel-Wilson \cite{rw} of the Morava K-theories of Eilenberg-MacLane spaces, as well as the $E_\infty$ obstruction theory developed by Goerss-Hopkins in \cite{gh}.  It is proven as Corollary \ref{snaith_cor}, below.

\subsection{Higher orientation for $K(n)$-local cohomology theories}

The map $\beta$ is traditionally used to define the notion of complex orientation for a cohomology theory; the associated formal group law is an important invariant of the cohomology theory, and central to the chromatic approach to homotopy theory.  We use the map $\rho_n$ to give a similar notion:

\begin{defn}

An \emph{$n$-orientation} of a $K(n)$-local ring spectrum $E$ is a class $x \in E^{\Sdet}(K(\Z_p, n+1))$ with the property that $\rho_n^*(x) \in E^{\Sdet}(\Sdet) = \pi_0 E$ is a unit.  

\end{defn}

Here, for an element $A \in \Pic_n$ and a space $X$ the notation $E^A(X)$ indicates the group
$$E^A(X) = [\Sigma^\infty X_+, A \otimes E];$$
this yields $E^m(X)$ when $A = S^m$.  Grading the associated groups over all of $\Pic_n$, we define
$$E^\bigstar(X) := \bigoplus_{A \in \Pic_n} E^A(X).$$

As for complex orientation, we show in Theorem \ref{orient_thm} that an $n$-orientation of $E$ gives a ring isomorphism $E^\bigstar(K(\Z_p, n+1)) \cong E^\bigstar[[x]]$.  The multiplication in $K(\Z_p, n+1)$ then yields a formal group law $F(x, y) \in E^\bigstar[[x, y]]$.  An $n$-oriented cohomology theory is said to be \emph{multiplicative} if its associated formal group law is of the form $F(x, y) = x + y + txy$ for a unit $t \in E^\bigstar$.

\begin{thm}

$R_n$ is the universal (i.e., initial) multiplicative $n$-oriented cohomology theory.

\end{thm}

This is Theorem \ref{s_o_thm} below.  Our proof is modelled on an argument of Spitzweck-{\O}stv{\ae}r \cite{spitz_ost} which gives a version of Snaith's theorem in the motivic setting.

This theorem allows us to identify $\Z_p$ as a summand of $\pi_0(R_n)$, and hence of $[\Sdet^{\otimes j}, R_n]$ for every $j$ by Theorem \ref{intro_snaith_thm}.  Furthermore, it implies that the action of the Adams operation $\psi^g$ on this summand is via the $j^{\rm th}$ power of the identity character.  Theorem \ref{main_intro_thm} follows from these facts by the long exact sequence in $[\Sdet^{\otimes j}, -]$-groups associated to a fibre sequence
$$\xymatrix{S \ar[r]^\eta & R_n \ar[r]^{\psi^g-1} & R_n}$$
which presents $S \simeq E_n^{h\G_n} = R_n^{h\Z_p^{\times}}$ as the homotopy equaliser of $\psi^g$ and $1$ (see section \ref{invert_section} for details).

\subsection{Algebraic K-theory and the chromatic redshift conjecture}

For an $A_\infty$ ring spectrum $A$, the space $\GL_1(A)$ of units (the invertible components in $\Omega^\infty A$) admits a delooping $B\GL_1(A)$.  If $A$ is in fact an $E_\infty$ ring spectrum, its multiplication equips $B\GL_1(A)$ with the structure of an infinite loop space.  We will write $K(A)$ for the algebraic $K$-theory spectrum of $A$.  There is a map of $E_\infty$ ring spectra
$$i: \Sigma^\infty B\GL_1(A)_+ \to K(A)$$
whose adjoint is given by the inclusion of $A$-lines into all cell $A$-modules (see \cite{abghr} for the construction on the level of $\infty$-categories; then \cite{em} gives the map of $E_\infty$ ring spectra).

In \cite{sati_wes}, with Hisham Sati, we constructed an $E_\infty$ map $\varphi_n: K(\Z, n+1) \to \GL_1 E_n$.  Delooping once and composing with $i$, we obtain a map of $E_\infty$ ring spectra 
$$i \circ B\varphi_n: \Sigma^\infty K(\Z, n+2)_+ \to K(E_n).$$
We may localise both the domain and range of this map with respect to $K(n+1)$; it is natural to ask about the behaviour of the composite map
$$\beta_{n+1} = i \circ B \varphi_n \circ \rho_{n+1} : \Sdet \to L_{K(n+1)} K(E_n)$$
(we note that the domain is the $K(n+1)$-local $\Sdet$).  

Our methods are not suitable to construct such a map for the case $n=0$ but it is instructive to consider this setting nonetheless.  We interpret $E_0$ as singular cohomology with $\Q_p$ coefficients; the resulting K-theory spectrum is the algebraic K-theory of $\Q_p$.  The map $\beta_1$, were it to exist, would be of the form $S^2 \to L_{K(1)} K(\Q_p)$; it should be considered as the image in the $K(1)$-local category of the \emph{Bott element}\footnote{This class doesn't exist unless $\mup$ is adjoined to $\Q_p$; however, its $p-1^{\rm st}$ power does exist in $K(\Q_p)$.  Alternatively, $\beta_1$ exists after smashing with $M\Z / p$.  See section \ref{redshift_section} for a similar phenomenon in our setting.} (considered in e.g., \cite{thomason} and \cite{mitchell}) in the $p$-adic algebraic $K$-theory of $\Q_p$
. 
\begin{thm} \label{red_thm}

For $p>3$, multiplication by $\beta_{n+1} \in \pi_{\Sdet} L_{K(n+1)}K(E_{n})$ is an equivalence.  Therefore
$$L_{K(n+1)}K(E_{n}) \simeq L_{K(n+1)}K(E_{n})[\beta_{n+1}^{-1}],$$
and so the map $i \circ B\varphi_n$ makes $L_{K(n+1)} K(E_n)$ an $R_{n+1}$-algebra spectrum.

\end{thm}

This is proven as Corollary \ref{non_nilp_cor} below.  It provides some evidence for the chromatic redshift conjecture, as enunciated in Conjecture 4.4 of \cite{ausoni-rognes}. The spectrum of the theorem is the algebraic K-theory spectrum of a prominent $K(n)$-local spectrum.  After localisation at $K(n+1)$, we have shown it to support an algebra structure one chromatic level higher.  

It is worth pointing out that this result is entirely $K(n+1)$-local.  The corresponding statement for $n=0$ -- that multiplication by the Bott element as considered by Thomason is a $K(1)$-local equivalence -- follows naturally from the fact that it descends to the familiar Bott class in the complex $K$-theory spectrum.  A deeper statement in fact holds for $n=0$: the Bott element exists prior to localisation, and multiplication by it is not nilpotent.  Furthermore, inversion of this element \emph{defines} the $K(1)$-localisation after smashing with an appropriate Moore spectrum (see \cite{mitchell}).

For $n=1$, Ausoni has constructed in \cite{ausoni} a class that he calls the \emph{higher Bott element}, $b \in V(1)_{2p+2}(K(E_1))$, and has verified that it, too, is not nilpotent using topological cyclic homology techniques.  Again, this construction occurs at a stage prior to $K(2)$-localisation.  It follows from his Theorem 1.1 that for $p>5$, after localisation, $b^{p-1}$ is a unit multiple of a $V(1)$-Hurewicz image of the element $\beta_2^{p-1}$ in 
$$[S^{2p^2-2}, L_{K(n)}(V(1) \wedge K(E_1))] \cong [\Sdet^{p-1}, L_{K(n)}(V(1) \wedge K(E_1))]$$
(this isomorphism follows from the same arguments as Proposition \ref{moore_prop}).  In contrast to the intricate computational methods of \cite{ausoni}, our proof of the periodicity of $\beta_{n+1}$ is achieved by detecting it modulo $p$ as multiplication by an invertible element of the Picard-graded homotopy of the $K(n+1)$-local Moore spectrum.

However, an important caveat to Theorem \ref{red_thm} is that it does not actually show that the localisation $L_{K(n+1)}K(E_{n})$ is nonzero.  When $n=0$, the fact that $L_{K(1)} K(\Q_p) \neq 0$ is visible in \cite{hess_mad}, where it is shown that $K_*(\Q_p)$ contains the image of $J$ as a factor.  When $n=1$, the nonvanishing of $L_{K(2)} K(E_1)$ follows from Ausoni's explicit computations of $V(1)_* K(E_1)$.

\subsection{Acknowledgements}

It is a pleasure to thank Mark Behrens, Ethan Devinatz, Bj{\o}rn Ian Dundas, Drew Heard, Jacob Lurie, Peter May, Eric Peterson, and Charles Rezk for a number of very helpful conversations.  I refer the reader to Peterson's very interesting \cite{peterson_annular} for an alternative approach to some of the constructions studied here.  Several of the ideas in this paper grew from the results in \cite{sati_wes}, joint with Hisham Sati; I wish to thank him for an excellent collaboration.  I am very grateful to Daniel Davis and Tyler Lawson, whose careful reading of a first draft of this paper has greatly improved it.  I would also like to thank the referee for a number of important clarifications and improvements on proofs.  Lastly, even a cursory read of this paper will reflect the fact that it is deeply influenced by the papers \cite{rw, hms, abghr, spitz_ost}.  I am happy to acknowledge support for this research from the Australian Research Council under the Discovery Project and Future Fellowship schemes, as well as the National Science Foundation, via DMS-1406162.

\section{Generalities on the $K(n)$-local category}

\subsection{Morava $K$ and $E$-theories}

A central object of the chromatic program is the Honda formal group law $\Gamma_n$ of height $n$ over the finite field $\F_{p^n}$; the $p$-series of $\Gamma_n$ is $[p](x) = x^{p^n}$.  We will write $K$ or $K_n$ for the 2-periodic Landweber-exact cohomology theory that supports this formal group law; thus the coefficients of $K$ are $\pi_* K = \F_{p^n}[u^{\pm 1}]$. This is closely related to the ``standard" Morava K-theory, $K(n)$, whose homotopy groups are $\F_p[v_n^{\pm 1}]$, with $|v_n| = 2p^n-2$.

The (Lubin-Tate) universal deformation of $\Gamma_n$ is defined over the ring $\W(\F_{p^n})[[u_1, \dots, u_{n-1}]][u^{\pm 1}]$, where $\W(\F_{p^n})$ is the ring of Witt vectors over $\F_{p^n}$.  There is a cohomology theory $E_n$ -- \emph{Morava $E$-theory} -- whose homotopy groups are given by that ring, and support this formal group law.  The theorem of Goerss-Hopkins-Miller \cite{gh} equips $E_n$ with the structure of an $E_\infty$ ring spectrum.

If we write $\m$ for the maximal ideal $\m = (p, u_1, \dots, u_{n-1}) \subseteq \pi_*(E_n)$, these cohomology theories are related via:
$$K^*(X) \cong (E_n / \m)^*(X) \cong \F_{p^n} \otimes_{\F_p} K(n)^*(X)[u] / (u^{p^n-1} - v_n),$$
and similarly in homology.  Notice that the representing spectrum $K$ is equivalent to a bouquet of copies of $K(n)$.  Thus its Bousfield class is the same as that of $K(n)$.  

Recall that $\F_{p^n}^{\times}$ consists of $p^n-1^{\rm st}$ roots of unity.  Factoring $p^n-1 = (p-1)(1 + p + \dots + p^{n-1})$, we note that $2(p-1)$ divides $p^n-1$ when $p$ is odd and $n$ is even.  Choose a $p-1^{\rm st}$ root of $(-1)^{n-1}$, $\xi \in \F_{p^n}$, to be:
$$\xi := \left\{ \begin{array}{ll}
1, & \mbox{$n$ odd} \\
\mbox{a primitive $2p-2^{\rm nd}$ root of $1$}, & \mbox{$n$ even}.
\end{array} \right.$$

\subsection{The Morava stabiliser group}

We will briefly summarise some preliminaries regarding the Morava stabiliser group.  We refer the reader to \cite{morava} and \cite{ghmr_resolution} for more complete discussions of this material.

We will use the notation $\bS_n := \Aut_{\F_{p^n}}(\Gamma_n)$ to denote the group of automorphisms of $\Gamma_n$.  Then $\G_n$ is defined as the semidirect product
$$\G_n := \bS_n \rtimes \Gal(\F_{p^n}/\F_p).$$
The Goerss-Hopkins-Miller theorem lifts the defining action $\G_n$ on $\Gamma_n$ to an action on the spectrum $E_n$ through $E_\infty$ maps.  The work of Devinatz-Hopkins \cite{dev_hop}, Davis \cite{davis}, and Behrens-Davis \cite{bd} defines homotopy fixed point spectra with respect to closed subgroups of $\G_n$ and constructs associated descent spectral sequences.

We write $\OO_n$ for the noncommutative ring
$$\OO_n = \W(\F_{p^n})\langle S \rangle / (S^n-p, Sa = a^{\sigma} S)$$
where $\sigma$ denotes a lift of the Frobenius on $\F_{p^n}$ to the ring $\W(\F_{p^n})$ of Witt vectors.  Then $\OO_n \cong \End_{\F_{p^n}}(\Gamma_n)$is the ring of endomorphisms of $\Gamma_n$, and so $\bS_n \cong \OO_n^{\times}$.

Now, $\bS_n$ naturally acts on $\OO_n$ (by right multiplication) through left $\W(\F_{p^n})$-module homomorphisms, and so defines a map $\bS_n \to \GL_n(\W(\F_{p^n}))$, since $\OO_n$ is free of rank $n$ over $\W(\F_{p^n})$, with basis $\{ 1, S, \dots, S^{n-1} \}$.  It turns out that the determinant of elements coming from $\bS_n$ actually lie in $\Z_p^{\times}$, instead of $\W(\F_{p^n})^{\times}$; this gives a homomorphism
$$\det: \bS_n \to \Z_p^{\times}.$$
Note that the Frobenius generator $F$ of $\Gal(\F_{p^n}/\F_p)$ acts on the division algebra $\OO_n \otimes \Q$ by conjugation by $S$.  Therefore, we may extend $\det$ to $\detpm: \G_n \to \Z_p^{\times}$ by sending $F \in \Gal(\F_{p^n}/\F_p)$ to $(-1)^{n-1}$; the action of $F$ on $\bS_n$ is such that this is indeed a homomorphism.  We will write $\SGn$ for the kernel of this homomorphism.  

The sign character is somewhat unusual (e.g., it does not appear in \cite{gross-hopkins} or \cite{ghmr_resolution} and further work of these authors), but is necessary in order to deal with a troublesome sign in the Ravenel-Wilson computations of the Verschiebung in the Morava K-theory of Eilenberg-MacLane spaces.  This is what requires the introduction of the root of unity $\xi$ in section \ref{rw_section} and is further reflected in the proof of Proposition \ref{morava_mod_prop}.

We  recall that there is an isomorphism $\Z_p^{\times} = \mu_{p-1} \times (1 +p\Z_p)^{\times} \cong \mu_{p-1} \times \Z_p$.  Composing the determinant with the projection onto the second factor yields the reduced determinant\footnote{Since $\pm 1 \in \mu_{p-1}$, this is the same as the reduced determinant of, e.g., \cite{ghmr_resolution}, so our $\G_n^1$ is the same as in that article.  However, the homomorphism to $\mu_{p-1}$ in equation (\ref{SES_eqn}) differs.} $\tdet: \G_n \to \Z_p$.  If we write $\G_n^1$ for the kernel of $\tdet$, then there is a split short exact sequence
\begin{equation}\label{SES_eqn}
1  \to \SGn \to \G^1_n \to \mu_{p-1} \to 1
\end{equation}

Furthermore, one may define a homomorphism $z: \Z_p^{\times} \to \bS_n$ by $a \mapsto a + 0 \cdot S + \dots + 0 \cdot S^{n-1}$.  The image is central, and the composite $\detpm \circ z: \Z_p^{\times} \to \Z_p^{\times}$ is $a \mapsto a^n$.  In particular, if $n$ is coprime to both $p$ and $p-1$, this map is an isomorphism, yielding a splitting $\G_n \cong \SGn \times \Z_p^{\times}$.  Similarly, we get a splitting $\G_n \cong \G_n^1 \times \Z_p$.

\subsection{The Picard group}

We recall that the category of $K(n)$-local spectra is symmetric monoidal; the tensor product
$$A \otimes B := L_{K(n)}(A \wedge B)$$
is the $K(n)$-localisation of the smash product of $A$ and $B$.  The unit is the $K(n)$-local sphere $S: = L_{K(n)}(S^0)$.  

An \emph{invertible spectrum} $A$ is invertible with respect to this tensor product; that is, there exists another spectrum $B$ with $A \otimes B \simeq S$.  The set of weak equivalence classes of invertible spectra forms an abelian group under tensor product, the \emph{Picard group} of the $K(n)$-local category, denoted $\Pic_n$. 

It was shown in \cite{hms} that $A$ is invertible if and only if $K(n)_*(A)$ is free of rank one over $K(n)_*$.  Consequently there is a well-defined \emph{dimension} $\dim(A) \in \Z / 2(p^n-1)$ which records the degree of a generator of $K(n)_*(A)$.

For spectra $A$, $B$, $F(A, B)$ will denote the function spectrum of maps from $A$ to $B$.  If $A$ is an element of $\Pic_n$, it is easily seen that the Spanier-Whitehead dual $F(A, S) = A^{-1}$ is the inverse to $A$: the evaluation map $A \otimes F(A, S) \to S$ is an equivalence.  

\subsection{$\Pic_n$-graded invariants}

In \cite{hms} it was suggested that one should take a larger view of homotopy groups in the $K(n)$-local category, and grade homotopy by the Picard group.  To invent some notation, we will write $\pi_A(X) := [A, X]$, for $K(n)$-local spectra $X$ and invertible $A$; then $\pi_{S^m}(X) = \pi_m(X)$.

We propose that the same approach be used to index\footnote{The point is raised in \cite{hovey-strickland} that grading homotopy (and hence other such invariants) by the Picard group does not necessarily yield a well-behaved ring structure on the resulting direct sum.  In practice, however, we will only ever consider indexing on integral powers of spheres and $\Sdet$, for which their Theorem 14.11 is more than sufficient.} the coefficients of more general (co)homology theories.  

\begin{defn}

Let $X$ and $E$ be $K(n)$-local spectra.  For an invertible spectrum $A$, define 
$$\begin{array}{cccc}
E_A(X) := [A, X \wedge E], & E^\vee_A(X) := [A, X \otimes E], & {\rm and} & E^A(X):= [X, A \otimes E]. 
\end{array}$$
These are the $\Pic_n$-graded \emph{$E$-homology}, \emph{completed $E$-homology}, and \emph{$E$-cohomology}.

\end{defn}

We note that $E_A(X)$ need not be isomorphic to the $E_A^\vee(X)$, even for $A = S^m$.  Indeed, when $E=E_n$ is Morava $E$-theory, the completed homology theory is the \emph{Morava module} of $X$, which may differ from $(E_n)_m(X)$ when $X$ is an infinite complex.  In this paper, we will largely be concerned with $E^\vee_A(X)$, rather than $E_A(X)$.  

To be completely balanced, we should also consider the cohomological functor $[X, A \wedge E]$ (and perhaps decorate the functor $E^A(X)$ above with a $\vee$).  However, $[X, A \wedge E] = [X, A \otimes E]$ when $A$ is a  sphere (since $S^m \wedge E = \Sigma^m E$ is already $K(n)$-local).  Furthermore, as in the homological setting, we will largely concern ourselves with $K(n)$-local constructions, making the non-local $A \wedge E$ awkward.

To distinguish from the standard indexing (by elements $m \in \Z$, corresponding to $S^m$), we will write
$$
E^\vee_\bigstar(X) := \bigoplus_{A \in \Pic_n} E^\vee_A(X), \; \mbox{ but } \;\;  E^\vee_*(X) := \bigoplus_{n \in \Z} E^\vee_n(X)
$$
and similarly for cohomology.

\begin{exmp} \label{k_invert_exmp}

Since $K$ satisfies a strong K\"unneth theorem,
$$K^A(X) = [X, A \otimes K] = [X \otimes A^{-1}, K] = K^0(X \otimes A^{-1}) = (K^*(X) \otimes_{K_*} K^{*}(A^{-1}))_0$$
If $A$ has dimension $d$, then $A^{-1}$ has dimension $-d$, so we conclude that $K^A(X) \cong K^d(X)$.  So grading $K^\bigstar$ by the full Picard group does not recover any more information than grading by the integers.  

The same argument works if one replaces $K$ with $E_n$, since $E_n^{*}(A^{-1})$ is free of rank $1$ over $E_n^*$, giving a collapsing K\"{u}nneth spectral sequence. However, the result is far from true for all (co)homology theories, most notably stable homotopy.

\end{exmp}

Additionally, we note the following, whose proof is immediate:

\begin{prop}

If $A$ is invertible and $E$ a $K(n)$-local ring spectrum, then $E^\vee_\bigstar(A)$ and $E^\bigstar(A)$ are free $E_\bigstar := E_\bigstar(S)$-modules of rank one. 

\end{prop}

Note that, by the (collapsing) K\"unneth spectral sequence, this implies that for any spectrum $X$, the natural map
$$E_\bigstar(A) \otimes_{E_{\bigstar}} E_\bigstar(X) \to E_\bigstar(A \otimes X)$$
is an isomorphism (and similarly for the completed homology and cohomology).

\subsection{Localising ring spectra} \label{loc_section}

If $X$ is a $K(n)$-local ring spectrum with multiplication $\mu$, $A \in \Pic_n$, and $f:A \to X$ an element of $\pi_A(X)$, one can localise $X$ away from $f$ in the same fashion as is usually done for homotopy elements.

\begin{defn}

Define $X[f^{-1}]$ to be the $K(n)$-localisation of the telescope (homotopy colimit) of the diagram 
$$\xymatrix{X \ar[r]^-{m_f} & A^{-1} \otimes X \ar[r]^-{1 \otimes m_f} & A^{-1} \otimes (A^{-1} \otimes X) \ar[r]^-{1 \otimes 1 \otimes m_f} & \dots.}$$
Here $m_f$ is the composite
$$\xymatrix{X = A^{-1} \otimes A \otimes X \ar[r]^-{1 \otimes f \otimes 1} & A^{-1} \otimes X \otimes X  \ar[r]^-{1 \otimes \mu} & A^{-1} \otimes X}$$
where $\mu$ is the multiplication in $X$.

\end{defn}

The $K(n)$-localisation in the definition is necessary because it is not clear that the telescope of $m_f$, being a homotopy colimit of $K(n)$-local spectra, is itself local.

\begin{prop} \label{e-infty_loc_prop}

If $X$ is an $E_\infty$-ring spectrum, then so too is $X[f^{-1}]$.

\end{prop}

\begin{proof}

This is based on section VIII of \cite{ekmm}, which shows that $KU = ku[\beta^{-1}]$ is an $E_\infty$-ring spectrum; we thank the referee for directing us to it.  All references in this proof are to \cite{ekmm}; this argument differs from theirs only in the $K(n)$-local aspect.  The telescope of $m_f$ is a (right) $X$-module; Proposition V.2.3 gives it also the weak structure of $X$-ring spectrum in the classical homotopy theoretic sense.  These descend through the $K(n)$-localisation to give structures on $X[f^{-1}]$ using, e.g., Theorem VIII.1.6 for the module structure and the fact that Bousfield localisation is monoidal up to homotopy for the ring structure.

This allows us, for any $K(n)$-local $X$-module $M$, to consider the map
$$\lambda: M \cong X \otimes_X M \to X[f^{-1}] \otimes_X M$$
given by the unit of $X[f^{-1}]$.  This has the property that $\lambda \otimes_X X[f^{-1}]$ is an equivalence, since the multiplication $X[f^{-1}] \otimes_X X[f^{-1}] \to X[f^{-1}]$ is an equivalence (this is the $K(n)$-localisation of Lemma V.1.15).  Thus $\lambda$ is a $X[f^{-1}]$-equivalence, and so $X[f^{-1}] \otimes_X M$ is a $X[f^{-1}]$-Bousfield localisation of $M$.  Applying this to $M = X$, which is an $X$-algebra, and invoking Theorem VIII.2.2, we see that $X[f^{-1}]$ is an $E_\infty$ $X$-algebra.

\end{proof}

Standard properties of homotopy colimits then give: 

\begin{prop} \label{loc_prop}

The natural map $X \to X[f^{-1}]$ induces an isomorphism 
$$(K_\bigstar X) [f^{-1}] \to K_\bigstar(X [f^{-1}] ).$$
where $f$ is regarded as an element of $K_A(X) = [A, X \otimes K]$ by smashing with the unit of $K$.  This induces an isomorphism
$$(K_* X) [f_*^{-1}] \to K_*(X [f^{-1}] ).$$
where $f_* \in K_{\dim(A)}(X)$ is the image of a generator of $K_*(A)$ under $f$.  The same holds with $K$ replaced by $E_n$.

\end{prop}

Write $\eta$ for the unit in $X$.  By construction, there is an equivalence $\mu \circ (f \otimes 1): A \otimes X[f^{-1}] \to X[f^{-1}]$; let $g$ be the inverse equivalence.  Define $f^{-1} : A^{-1} \to X[f^{-1}]$ as
$$\xymatrix@1{A^{-1} = A^{-1} \otimes S \ar[r]^-{1 \otimes \eta} & A^{-1} \otimes X[f^{-1}] \ar[r]^-{1 \otimes g} & A^{-1} \otimes A \otimes X[f^{-1}]  = X[f^{-1}] }$$
By construction $\eta = f \cdot f^{-1} : S = A \otimes A^{-1} \to X[f^{-1}]$; thus $f$ is indeed invertible in $\pi_\bigstar(X[f^{-1}])$.

\subsection{Automorphism groups}

For a spectrum $A$, we will write $\eend(A) := F(A, A)$ for the function spectrum of maps from $A$ to itself.  This is an associative ring spectrum under composition of functions, and so $\pi_0 \eend(A, A) = [A, A] $ forms a ring.

\begin{defn}

Write $\Aut(A) \subseteq \End(A) :=  \Omega^\infty \eend(A)$ for the union of components whose image in the ring $\pi_0 \eend(A, A)$ is invertible.  

\end{defn}

Note that the multiplication in $\eend(A)$ equips $\Aut(A)$ with the structure of a grouplike $A_\infty$ monoid.  We also note that if $A = S$, the adjoint of the identity on $A$ defines an equivalence of ring spectra $\eend(S) \simeq S$, and so $\Aut(S)$ is identified as the $A_\infty$ monoid of units in the ring spectrum $S$, $\GL_1(S)$:
$$\Aut(S) = \GL_1 (S) \subseteq \Omega^\infty  S$$

The following is immediate:

\begin{prop}

If $A$ is invertible and $B$ any spectrum, tensoring with the identity of $A$ gives an equivalence 
$$\id_A \otimes - : \eend(B) \to \eend(A \otimes B).$$
Passing to infinite loop spaces yields a natural equivalence of $A_\infty$ monoids $\Aut(B) \to \Aut(A \otimes B)$.

\end{prop}

In particular, taking $B= S$, we conclude

\begin{cor} \label{aut_cor}

If $A$ is invertible, smashing with the identity of $A$ defines an equivalence of $A_\infty$ monoids 
$$\GL_1(S) \to \Aut(A).$$

\end{cor}

\noindent Loosely speaking, $A$ becomes a $\GL_1(S)$-spectrum by its action on the left tensor factor of $S \otimes A = A$.

\subsection{Thom spectra} \label{thom_section}

Let $X$ be a topological space, and $\zeta: X \to B \GL_1(S)$ a continuous map.  Ando, et al. define the Thom spectrum $X^\zeta$ in \cite{abghr} as
$$X^\zeta := \Sigma^\infty P_+ \wedge_{\Sigma^\infty \GL_1(S)_+} S$$
where $P$ is the principal $\GL_1(S)$-bundle over $X$ defined by $\zeta$.  Using Corollary \ref{aut_cor}, we may modify and extend this definition in the $K(n)$-local category:

\begin{defn}

For an invertible spectrum $A$, define the \emph{Thom spectrum} $X^{A\zeta}$ as  
$$X^{A\zeta} := \Sigma^\infty P_+ \otimes_{\Sigma^\infty \GL_1(S)_+} A = L_{K(n)}(\Sigma^\infty P_+ \wedge_{\Sigma^\infty \GL_1(S)_+} A).$$

\end{defn}

A remark on notation: we have chosen to present this multiplicatively ($X^{A\zeta}$ instead of $X^{A + \zeta}$) to be consistent with the multiplicative nature of the (smash) product in the Picard group.  This of course conflicts with the common additive notation $X^{n+\zeta}$ when $A = S^n$; we hope no confusion will result.

Note that even when $A=S$, this differs slightly from the definition in \cite{abghr} in that we have $K(n)$-localised the Thom spectrum.  This extension of the definition of Thom spectra to have any invertible spectrum as ``fibre" is convenient, but not very substantial; noticing that the action of $\GL_1(S)$ is on the $S$ factor in $A = S \otimes A$, one can easily show:

\begin{prop}

$X^{A\zeta} = X^\zeta \otimes A$.

\end{prop}

The composite 
$$\xymatrix@1{
X \times X \ar[r]^-{pr_1} & X \ar[r]^-\zeta&  B \GL_1(S)
}$$
of $\zeta$ with the projection onto the first factor defines a Thom spectrum over $X \times X$, which may be identified with $X^{A \zeta} \otimes X_+$.  Furthermore, the diagonal map $\Delta: X \to X \times X$ is covered by a map of Thom spectra 
$$D: X^{A \zeta} \to X^{A \zeta} \otimes X_+$$
since $pr_1 \circ \Delta = \id_X$.  The map $D$ is the \emph{Thom diagonal}, and for any $K(n)$-local ring cohomology theory $E$, makes $E^\bigstar(X^{A\zeta})$ into a $E^\bigstar(X)$-module.  

Similarly, for each point $x \in X$, the inclusion $\{x\} \subseteq X$ is covered by a ``fibre inclusion" 
$$i_x: A = X^{A\zeta}|_{\{x\}} \to X^{A\zeta}.$$

\begin{defn}

A class $u \in E^A(X^{A \zeta})$ is a \emph{Thom class} if, for every $x \in X$ its restriction along $i_x: A \to X^{A \zeta}$, $i_x^*(u) \in E^A(A) \cong E_0$ is a unit.

\end{defn}

\begin{prop} \label{thom_prop}

If $X^{A \zeta}$ admits a Thom class $u$ for $E$, then $E^\bigstar (X^{A \zeta})$ is a free $E^\bigstar(X)$-module of rank one, generated by $u$.

\end{prop}

When $A=S$, this is classical, and is realised in homology by the map:
$$\xymatrix{
E \otimes X^{A \zeta} \ar[r]^-{1 \otimes D} & E \otimes X^{A \zeta} \otimes X_+ \ar[r]^-{1\otimes u \otimes 1} & E \otimes E \otimes A \otimes X \ar[r]^-{\mu \otimes 1} & E \otimes A \otimes X.
}$$
See, e.g., \cite{mah_ray} or \cite{abghr}, Prop 5.43.  For general $A$, this is obtained from that fact and the $\Pic_n$-graded isomorphism 
$$E^\bigstar(X^{A \zeta}) = E^\bigstar(X^{\zeta} \otimes A)  \cong E^\bigstar(X^{\zeta}) \otimes_{E_\bigstar} E_\bigstar A.$$

\section{Eilenberg-MacLane spaces}

\subsection{Recollections from Ravenel-Wilson} \label{rw_section}

Working stably, after $p$-completion, we note the equivalences\footnote{Here $\mup \cong \varinjlim \Z / p^j$ is the group of $p^{\rm th}$ power roots of unity in, e.g., $\C$.}
\beqn \label{limit_eqn}
\Sigma^\infty K(\mup, n)_+ = \varinjlim \Sigma^\infty K(\Z / p^j, n)_+ \simeq \Sigma^\infty K(\Z_p, n+1)_+ \simeq \Sigma^\infty K(\Z, n+1)_+
\eeqn
(The first equivalence uses the Bockstein). This will be our main object of study:

\begin{defn}

Let $X = X_n$ denote the $K(n)$-localisation of the unreduced suspension spectrum of $K(\Z_p, n+1)$,
$$X_n = L_{K(n)} \Sigma^\infty K(\Z_p, n+1)_+.$$

\end{defn}

We recall from Ravenel-Wilson \cite{rw} the Morava K-theory of this spectrum:
$$\begin{array}{ccc}
K(n)^* K(\Z_p, n+1) = K(n)_*[[x]], & \mbox{and} & K(n)_* K(\Z_p, n+1)  = \bigotimes_{k\geq 0} R(b_k).\end{array}$$
Here $|x| = 2g(n)$, where we define $g(n) := \frac{p^n-1}{p-1}$.  In the notation of section 12 of \cite{rw}, $x$ corresponds to the class $x_S$ for $S = (1, 2, \dots, n-1)$.  Also, for each integer $k \geq 0$, $R(b_k)$ is the ring 
$$R(b_k) := K(n)_*[b_k] / (b_k^p - (-1)^{n-1} v_n^{p^k} b_k) = \F_p[x, v_n^{\pm}] / (b_k^p - (-1)^{n-1} v_n^{p^k} b_k),$$
where the class $b_k$ has dimension $2p^k g(n)$ and is dual to  $(-1)^{k(n-1)} x^{p^k}$.  The notation $b_k$ is our abbreviation for Ravenel-Wilson's $b_J$, with $J = (nk, 1, 2, \dots, n-1)$.

There are similar results for $K(\Z / p^j, n)$:
$$\begin{array}{ccc}
K(n)^* K(\Z / p^j, n) = K(n)_*[x]/x^{p^j}, & \mbox{and} & K(n)_* K(\Z / p^j, n) = \bigotimes_{k= 0}^{j-1} R(b_k)
\end{array}$$
We have normalised these classes to be consistent across the limit in (\ref{limit_eqn}).  This is not quite consistent with the notation of section 11 of \cite{rw}; there $K(n)_* K(\Z / p^j, n)$ is presented as being generated by classes $a_I$, with $I = (nk, n(j-1)+1, n(j-1)+2, \dots, n(j-1) +n-1)$ with $0\leq k < j$.  This $a_I$ differs from $b_k$ by a power of $v_n$.

In the extension
$$K^*X_n = \F_{p^n} \otimes_{\F_p} K(n)^*(X_n)[u] / (u^{p^n-1} - v_n),$$
use the $2p-2^{\rm nd}$ root of unity, $\xi$, to define a new (degree 0) coordinate $y:= \xi x u^{g(n)}$; then
$$K^*(X_n) = K_*[[y]]$$
We may similarly normalise the $K$-homology; setting $c_k = (-1)^{k(n-1)}\xi^{-p^k} b_k u^{-p^k g(n)} =\xi^{-1} b_k u^{-p^k g(n)}$ we have
$$K_*(X_n) = K_*[c_0, c_1, \dots]/(c_k^p-c_k).$$
and $\langle c_k, y^{p^j} \rangle = \delta_k^j$.

\begin{prop} \label{rw_prop}

The multiplication on $X_n$ equips $K^* X_n$ with the structure of a formal group over $K_*$ which is isomorphic to the formal multiplicative group, $\G_m$.

\end{prop}

\begin{proof}

The computations described above equip $K^* X_n$ with a coordinate $y$, and the associative and unital multiplication on $X_n$ coming from the H-space structure on $K(\Z_p, n+1)$ defines a formal group law $F$ on $K^* X_n = K_*[[y]]$.

We may compute the $p$-series of this formal group law using \cite{rw}.  The relevant fact is that the Verschiebung, well-defined\footnote{Because $\F_p = K(n)_* / (v_n - 1)$ is a perfect field, while $K(n)_*$ is not, $V$ is not naturally defined on $K(n)^*(X)$, but on its cyclically graded analogue $\overline{K(n)}^*(X) = K(n)^*(X) / (v_n-1)$.} up to powers of $v_n$,  satisfies $V(x) = (-1)^{n-1} x$.  Thus 
$$[p](x) = FV(x) = (-1)^{n-1} v_n^{-1} x^p$$
Consequently $[p](y) = y^p$, and so $F$ has height 1.  Thus there is an isomorphism $f: F \cong \G_m$ over $\overline{\F}_p$.  In fact $f$ is defined over $\F_p$, for the equation
$$f(y^p) = f([p]_F(y)) = [p]_{\G_m}(f(y)) = f(y)^p$$
implies that if $f(y) = \sum a_i y^i$, then the coefficients $a_i$ satisfy $a_i^p =a_i$, so must lie in $\F_p$.  Note, however, that the formal group on $K(n)^*X_n$ is not isomorphic (over $\F_p$) to $\G_m$ when $n$ is even, since the root of unity $\xi$ used in defining $y$ is not present in $\F_p$ in that case.

\end{proof}

We note that this computation allows one to formally define the $a$-series $[a](y) \in K_*[[y]]$ for any element $a \in \Z_p$.  

\begin{defn}

Write $a$ in its $p$-adic expansion as $a = a_0 + a_1 p + a_2 p^2 + \dots$, where $0 \leq a_i <p$.  Then 
$$[a](y) := [a_0](y) +_F [a_1]([p](y)) +_F [a_2]([p^2](y)) +_F \dots$$
Here, $+_F$ is addition according to the formal group law on $K_*[[y]]$.  

\end{defn}

The formula gives a well-defined series, since $\deg ([p^n](y)) = p^n$ (or equivalently in this case: the order of vanishing at $0$) grows with $n$. 

\subsection{Group actions}

The $p$-adic integers $\Z_p$ are a topologically cyclic group with generator $1 \in \Z_p$; that is, the subgroup generated by $1$ (i.e., $\Z$) is dense in $\Z_p$.  We will have occasion to write elements of $\Z_p$ in multiplicative notation; then we will write $h$ for the generator 1.  

Furthermore, for $p>2$, 
$$\Z_p^{\times} \cong \mu_{p-1} \oplus (1+p \Z_p)^{\times},$$ 
and the latter factor is isomorphic to $\Z_p$.  Let $\zeta = \zeta_{p-1} \in \mu_{p-1}$ be a primitive $(p-1)^{\rm st}$ root of unity, i.e.,  a generator of $\mu_{p-1}$.  A generator for $(1+p \Z_p)^{\times}$ is $1+p$.  Then $\Z_p^{\times}$ is also topologically cyclic, with generator $g:= (\zeta, 1+p)$.  We note that $g \bmod p = \zeta \in \F_p^{\times}$.

The group $\Z_p^{\times}$ acts on $\Z_p$, and hence on $X_n = L_{K(n)} \Sigma^\infty K(\Z_p, n+1)_+$.  For an element $a \in \Z_p^{\times}$, we will denote by $\psi^a$ the map $\psi^a: X_n \to X_n$ given by the action of $a$.

\begin{prop} \label{action_prop}

The action of $\Z_p^{\times}$ on $K^*(X_n)$ is via ring homomorphisms, and is determined by $\psi^a (y) = [a](y)$.

\end{prop}

\begin{proof} 

This formula in fact holds for every $a \in \Z_p$.  Note that this homotopy commutes, for any $m \in \N$:
$$\xymatrix{
K(\Z_p, n+1) \ar[r]^-\Delta \ar[dr]_-{\psi^m} & K(\Z_p, n+1)^{\times m} \ar[d]^-{\rm mult} \\
 & K(\Z_p, n+1)
}$$
where $\Delta$ is the $m$-fold diagonal, and mult is $m$-fold multiplication.  The path along the upper right carries $y$ to $[m](y)$.  

Being defined by space-level maps, the action of $\Z_p$ on $K^*(X_n)$ may be regarded as a \emph{continuous} homomorphism $\Z_p \to \End(K_*[[y]])$.  Here, the topology on $\End(K^*(X))$ is compact-open with respect to the natural topology on $K_*(Y)$ defined in section 11 of \cite{hovey-strickland}.  In that topology, the map $\Z_p \subseteq [X, X] \to \End(K^*(X))$ is continuous.  The argument above indicates that it agrees with the action $a \cdot y = [a](y)$ when $a \in \N$.  Since the latter is also continuous, and $\N$ is dense in $\Z_p$, these two actions must agree.

\end{proof}

The group $\mu_{p-1} \cong \F_p^{\times}$ acts via multiplication on $\Z/p = \F_p$, and therefore on $K(\Z / p, n)$. 

\begin{prop} \label{finite_cor}

The action of $\mu_{p-1}$ on $K^*K(\Z / p, n) \cong K_*[y] / y^p$ is given by $\psi^\zeta (y^m) = \zeta^m y^m$.

\end{prop}

\begin{proof}

Since the action is space-level, it suffices to show that $\psi^\zeta (y) = \zeta y$.  To see this, we note that the following diagram commutes:
$$\xymatrix{
K(\Z /p, 1) \times K(\Z /p, 1)^{\times n-1} \ar[r]^-{\circ} \ar[d]_-{\psi^\zeta \times 1} & K(\Z / p, n) \ar[d]^-{\psi^\zeta} \\
K(\Z /p, 1) \times K(\Z /p, 1)^{\times n-1} \ar[r]_-{\circ} & K(\Z / p, n)
}$$
since both paths around the diagram represent $\zeta$ times the fundamental class in $H^n(K(\Z /p, 1)^{\times n}; \F_p)$.  Thus in $K(n)_*$, 
\begin{eqnarray*}
\psi^\zeta_*(b_0) & = & \psi^\zeta_*(a_{(0)} \circ a_{(1)} \circ \cdots \circ a_{(n-1)}) \\
 & = & \psi^\zeta_*(a_{(0)}) \circ a_{(1)} \circ \cdots \circ a_{(n-1)} \\
 & = & \zeta a_{(0)} \circ a_{(1)} \circ \cdots \circ a_{(n-1)} \\
 & = & \zeta b_0.
\end{eqnarray*}
The third equality uses the claim that $\psi^\zeta_*(a_{0}) = \zeta a_{(0)} \in K(n)_2 K(\Z / p, 1)$.  This follows from the fact that both classes are carried injectively by the Bockstein to the unique class in $K(n)_2 K(\Z_p, 2)$ which is $\zeta$ times the Hurewicz image of the fundamental class of $S^2 \to K(\Z_p, 2)$.  The claimed result then follows by duality.

\end{proof}

The commutativity of the diagram of exact sequences:
$$\xymatrix{
0 \ar[r] & \Z_p \ar[r]^-p \ar[d]_-{g} & \Z_p \ar[r]^-{\bmod p} \ar[d]_-{g} & \Z / p \ar[d]^-{\zeta} \ar[r] & 0 \\
0 \ar[r] & \Z_p \ar[r]_-p  & \Z_p \ar[r]_-{\bmod p}  & \Z / p \ar[r] & 0
}$$
yields a commutative diagram of fibrations
\beqn \label{diagram_eqn}
\xymatrix{
K(\Z / p, n) \ar[r]^-\beta \ar[d]_-{\psi^\zeta} & K(\Z_p, n+1) \ar[r]^-{\psi^p} \ar[d]^-{\psi^g} & K(\Z_p, n+1) \ar[d]^-{\psi^g} \\
K(\Z / p, n) \ar[r]^-\beta & K(\Z_p, n+1) \ar[r]^-{\psi^p} & K(\Z_p, n+1) 
}\eeqn
Thus, the Bockstein
$$\beta : L_{K(n)} \Sigma^\infty K(\Z / p, n)_+ \to L_{K(n)} \Sigma^\infty K(\Z_p, n+1)_+ = X_n$$
is $\Z_p^{\times}$-equivariant where the action of $\Z_p^{\times}$ on $L_{K(n)} \Sigma^\infty K(\Z / p, n)_+$ is via the reduction $\Z_p^{\times} \to \mu_{p-1}$.  In particular, we conclude that
$$\psi^g(y) \bmod y^p = \zeta y$$

\subsection{Splitting $K(\Z/p, n)$} 

In this section, we construct an invertible spectrum $Z$ and decompose $L_{K(n)} \Sigma^\infty K(\Z/p, n)_+$ into powers of $Z$.  It will turn out that $Z$ has order $p-1$ in $\Pic_n$, and is represented in the algebraic Picard group by the character $\chi$ given by the composite
$$\xymatrix{\G_n \ar[r]^-{\detpm} & \Z_p^{\times} \ar[r]^-{\bmod p} & \F_p^{\times} \ar[r]^-{T} & \Z_p^{\times}}$$
($T$ is the Teichm\"{u}ller character); that is, $E_*^{\vee}(Z) \cong E_* [\chi]$.

For any $p$-complete spectrum $Y$, there is an action (distinct from the action described in the previous section) of $\Z_p^{\times}$ on $Y$ where group elements acts in homotopy by multiplication; we will write the action of $\zeta$ simply as $\zeta$.

\begin{defn}

Define an endomorphism $\pi$ of $L_{K(n)} \Sigma^\infty K(\Z/p, n)_+$ by 
$$\pi = \frac{1}{p-1} \sum_{k=0}^{p-2} \zeta^{-k} (\psi^\zeta)^k = \frac{1}{p-1} \sum_{k=0}^{p-2} \zeta^{-k} \psi^{\zeta^k}$$

\end{defn}

We thank Tyler Lawson for this construction and the proof of the following proposition, which replaces an incorrect argument in a previous version of this paper.

\begin{prop} \label{small_split_prop}

The map $\pi$ is a homotopy idempotent.  It thus yields a splitting 
$$L_{K(n)} \Sigma^\infty K(\Z/p, n)_+ \simeq Z \vee Z^{\perp}$$
where the image of $Z$ in $K_*K(\Z/p, n)$ is the rank one $K_*$-subspace generated by $b_0$.

\end{prop}

\begin{proof}

The idempotence follows from the fact that $\psi^\zeta$ commutes with $\zeta$ and a brief computation.  The splitting comes in the standard way, defining $Z$ as the spectrum representing the image of $\pi$ inside the functor $[-, L_{K(n)} \Sigma^\infty K(\Z/p, n)_+]$.  Consequently, the image of $K_*(Z)$ is the image of $\pi_*$.  However,
$$\pi_*(b_0^\ell) = \frac{1}{p-1} \sum_{k=0}^{p-2} \zeta^{-k} \psi^{\zeta^k}(b_0^\ell) = \frac{1}{p-1} \sum_{k=0}^{p-2} \zeta^{-k} \zeta^{k\ell} b_0^\ell = \frac{1}{p-1} \sum_{k=0}^{p-2} \zeta^{k(\ell-1)} b_0^\ell$$
which is precisely $b_0$ when $\ell=1$, and $0$ otherwise.

\end{proof}

\begin{lem} \label{split_1_lem}

$Z$ is an element of $\Pic_n$.  Furthermore,
$$L_{K(n)} \Sigma^\infty K(\Z/p, n)_+ \simeq \bigvee_{k=0}^{p-1} Z^{\otimes k}.$$
Lastly, $Z^{\otimes p-1} \simeq S$.

\end{lem}

\begin{proof}

The first claim follows immediately from the previous Proposition.  Then the map $Z^{\otimes k} \to L_{K(n)} \Sigma^\infty K(\Z/p, n)_+$ induced by $k$-ary multiplication in $K(\Z/p, n)$ is an isomorphism in $K_*$ onto its image, which is the subspace generated by $b_0^k$.  Thus $L_{K(n)} \Sigma^\infty K(\Z/p, n)_+$ decomposes as indicated into a wedge of tensor powers of $Z$.

The same argument, applied to $k=p$ (along with the fact that $b_0^p=(-1)^{n-1} v_n b_0$) yields $Z^{\otimes p} \simeq Z$.  The last result follows by cancelling a factor of $Z$.

\end{proof}

Lastly, we note that $\pi \circ \psi^{\zeta} = \zeta \cdot \pi$; thus the action of $\psi^\zeta$ on $L_{K(n)} \Sigma^\infty K(\Z/p, n)_+$, when restricted to $Z$ is simply multiplication by $\zeta$.

\subsection{A description of $K_* K(\Z_p, n+1)$ in terms of functions on $\Z_p$}

Let $G$ be a profinite group and $R$ a topological ring.  Following \cite{hovey_ops}, we define the \emph{completed group ring on $G$}, $R[[G]]$ as the inverse limit of finite group rings
$$R[[G]] := \varprojlim_{U} R[G/U]$$
where $U$ ranges over open subgroups of $G$.  We recall that $R[[G]]$ admits the structure of an $R$-Hopf algebra where each $g \in G$ is grouplike: $\Delta(g) = g \otimes g$.   

One may similarly define the $R$-Hopf algebra $C(G, R)$ to be the ring of continuous $R$-valued functions on $G$.  The coproduct is the dual of multiplication in $G$.  Theorem 5.4 of \cite{hovey_ops} and the discussion that follows it give a proof that these two Hopf algebras are dual to each other over $R$: 
$$R[[G]] \cong \Hom_R(C(G, R), R).$$

\begin{rem}

The careful reader will note that in \cite{hovey_ops}, Hovey studies the \emph{twisted} completed group ring $R[[G]]$ in the setting where $G$ acts continuously on $R$; the multiplication in $R[[G]]$ is deformed by this action.  One similarly alters the coproduct (and right unit) in $C(G, R)$; the result is a Hopf \emph{algebroid}.  While of course essential to his goal of identifying $E_n^* E_n$ and its dual, for our purposes the untwisted analogues will suffice.  We note in particular that the algebra structure on $C(G, R)$ (and coalgebra structure on $R[[G]]$) is the same in the twisted and untwisted setting.

\end{rem}

When $G=\Z_p$ and $R = \F_p$, a natural family of functions is given as follows: if $m=m_0 + m_1 p + m_2 p^2 + \dots \in \Z_p$ with $m_i$ taken to be either $0$ or a $p-1^{\rm st}$ root of unity in $\Z_p$, define $f_k(m) = m_k \bmod p$.  Clearly, $f_k^p = f_k$.  In fact, the $f_k$ form a set of generators for $C(\Z_p, \F_p)$, and so
$$C(\Z_p, \F_p) = \F_p[f_0, f_1, f_2, \dots] / (f_k^p-f_k)$$
See, e.g., section 2 of \cite{ravenel_algebras} or 3.3 of \cite{hovey_ops}.  Lastly, if $k$ is any finite extension of $\F_p$, then the natural map $C(\Z_p, \F_p) \otimes_{\F_p} k \to C(\Z_p, k)$ is an isomorphism (this fact is perhaps more evident in the dual Hopf algebra).

Note that $\Z_p^{\times}$ acts on $\Z_p$, and hence on $R[[\Z_p]]$.

\begin{prop} \label{Z_p_prop}

There exists a $\Z_p^{\times}$-equivariant isomorphism of Hopf algebras
$$\phi: K_*[[\Z_p]] \to K^* K(\Z_p, n+1) \cong K_*[[y]],$$
which carries a topological generator $h$ of $\Z_p$ to $1+y$.

\end{prop}

\begin{proof}

The ring isomorphism $\phi$ is a basic fact about the Iwasawa algebra $\Z_p[[\Z_p]]$ (or, in this case, its reduction modulo $p$, and extension over $\F_{p^n}$).  The dual map $\phi^*$ satisfies 
$$\phi^*(c_k)(h^m) = \langle c_k, (1+y)^m \rangle = {m \choose p^k} \langle c_k, y^{p^k} \rangle = {m \choose p^k}$$
If we write $m$ in its $p$-adic expansion as $m=m_0 + m_1 p + m_2 p^2 + \dots$, then by Lucas' theorem, ${m \choose p^k} \bmod p = m_k$.   Therefore the map $\phi^*$ carries $c_k$ to $f_k \in C(\Z_p, K_*)$.  It is evidently a ring isomorphism,  and so $\phi$ is an isomorphism of Hopf algebras.

To see that the action of $\Z_p^{\times}$ is as claimed, we first note that $\Z_p^{\times}$ acts on $\Z_p$ through group homomorphisms; hence it acts on $K_*[[\Z_p]]$ through ring homomorphisms.  Thus it suffices to show that for $\gamma \in \Z_p^{\times}$, $\phi(\gamma \cdot h) = \gamma \cdot \phi(h)$.  The lefthand side is $\phi(h^\gamma) =\phi(h)^\gamma = (1+y)^\gamma$, whereas the right is $\gamma \cdot(1 + y) = 1 + [\gamma](y)$.  That these two are equal follows from the fact that $y$ is a coordinate on the formal multiplicative group.

\end{proof} 

We note that a more invariant description of this Hopf algebra may be given as follows:  Cartier duality gives an isomorphism $\Spec(K_0 K(\Z_p, n+1)) \cong \Hom(F, \G_m)$.  Further, $\Aut(\G_m) \cong \Z_p^{\times}$ acts by post-composition.

Proposition \ref{rw_prop} shows that the formal group $F$ associated to $K^0 K(\Z_p, n+1)$ is isomorphic to the multiplicative formal group $\G_m$.  If we \emph{choose} an isomorphism of $F$ with $\G_m$, $\Spec(K_0 K(\Z_p, n+1))$ is in turn identified with $\End(\G_m) \cong \Z_p$ as a $\Z_p^{\times}$ space, which yields the results above.  The choice of such an isomorphism corresponds to the choice of coordinate $y$ above.  Both points of view on this computation will be useful as we go forward.

\subsection{Inverting roots of unity; the spectrum $R_n$} 

Define $\alpha = \beta \circ i: Z \to X$ as the composite of the Bockstein with the natural map $i: Z \to \Sigma^\infty K(\Z/p, n)_+$.  Consider the localisation of $X$ at this element in the sense of section \ref{loc_section}:

\begin{defn}

Write $R_n$ for the $E_\infty$ ring spectrum $R_n := X[\alpha^{-1}]$.

\end{defn}

\begin{cor} \label{unit_cor}

The dual isomorphism of Proposition \ref{Z_p_prop} localises to a ring isomorphism
$$\phi^*: K_*(R_n) \to C(\Z_p^{\times}, K_*)$$

\end{cor}

\begin{proof}

Proposition \ref{loc_prop} implies that 
$$K_*(X[\alpha^{-1}]) = K_*(X)[\alpha_*^{-1}] = C(\Z_p, K_*)[ f_0^{-1}]$$
since the image of a generator of $K_*(Z)$ under $\alpha$ is $b_0 = \xi f_0u^{g(n)}$.  Note that $m \in \Z_p$ lies in $\Z_p^{\times}$ if and only if $f_0(m) = m \bmod p$ is invertible.  Therefore $C(\Z_p, K_*)[ f_0^{-1}]$ may be identified as $C(\Z_p^{\times}, K_*)$.

\end{proof}

We would like a $\Z_p^{\times}$-equivariant version of this result.  Recall from Proposition \ref{e-infty_loc_prop} that $X[\alpha^{-1}]$ may be presented as the Bousfield localisation $X[\alpha^{-1}] \simeq L_{X[\alpha^{-1}]} X$.  For each $k \in \Z_p$, the Bousfield localisation of $\psi^k: X \to X$ yields a map $\psi^k: X[\alpha^{-1}] \to X[\alpha^{-1}]$.

\begin{lem} \label{infty_lem}

For each $k \in \Z_p^{\times}$, the map $\psi^k: X[\alpha^{-1}] \to X[\alpha^{-1}]$ is an $E_\infty$ map.  Its  induced map in $K_* X[\alpha^{-1}] = C(\Z_p^{\times}, K_*)$ can be identified with translation of functions by $k$.  Finally, $\psi^k \circ \psi^{\ell} \simeq \psi^{k\ell}$.

\end{lem}

\begin{proof}

The last claim is simply that Bousfield localisation is functorial up to homotopy.  The first claim follows immediately from the fact that the localisation of $\psi^k$ is the lifted arrow
$$\xymatrix@1{
L_{X[\alpha^{-1}]} X \ar@{.>}[r]^-{\psi^k} & L_{X[\alpha^{-1}]} X \\
X \ar[u] \ar[r]_-{\psi^k} & X \ar[u]
}$$
provided by Theorem VIII.2.2 of \cite{ekmm}.  The second claim follows from Proposition \ref{Z_p_prop}, since that is how $\psi^k$ acts on $K_*(X) = C(\Z_p, K_*)$.

\end{proof}

A caveat is in order: the maps $\psi^k$ defined on $X[\alpha^{-1}]$ in the above fashion make sense for every $k \in \Z_p$, not just $\Z_p^{\times}$.  However, for any $k \in p\Z_p$, the second statement of the above Proposition cannot possibly hold, since translation by $k$ carries $\Z_p^{\times}$ off itself.  In fact, the only possible localisation of $\psi^k$ making the above diagram commute is zero in $K_*$.  However, $\psi^p$ is essential on $X$, as we will see in the next result.  Define $j: X \to R_n$ to be the natural map from a ring spectrum to its localisation.

\begin{prop} \label{weird_split_prop}

The map $j:X \to R_n$ splits; in fact, $j \vee \psi^p: X \to R_n \vee X$ is an equivalence.

\end{prop}

\begin{proof} 

Recall that $\psi^k$ acts on $K_* X = C(\Z_p, K_*)$, via multiplication by $k$ on $\Z_p$.  Consequently, 
$$(j \vee \psi^p)_*: K_* X = C(\Z_p, K_*) \to C(\Z_p^{\times}, K_*) \oplus C(\Z_p, K_*) = K_* (R_n \vee X)$$
is induced by the continuous bijection $\Z_p^{\times} \sqcup \Z_p \to \Z_p$ which is the natural inclusion on $\Z_p^{\times}$ and multiplication by $p$ on $\Z_p$ (i.e., $\Z_p = \Z_p^{\times} \sqcup p\Z_p$).  Since both spaces are compact Hausdorff, this is a homeomorphism, so it induces an isomorphism on $C(-, K_*)$.  Hence $j \vee \psi^p$ is a $K(n)$-local equivalence.

\end{proof}

In light of Snaith's Theorem, for $n=1$ this says that $KU^{\wedge}_p$ is a split summand of $L_{K(1)} \Sigma^\infty \C P^\infty_+$.  This  can alternatively be proven as follows: in \cite{segal}, Segal proves that $BU$ is a summand of $Q \C P^\infty$ (although not as infinite loop spaces).  Apply the Bousfield-Kuhn $\Phi_1$-functor to the splitting map $BU \to Q \C P^\infty$ to get a one-sided inverse $L_{K(1)} KU \to L_{K(1)} \Sigma^\infty \C P^\infty_+$ to $j$.

\subsection{Invertible spectra as homotopy fibres} \label{invert_section}

\begin{prop} \label{invertible_prop}

Let $\gamma \in \pi_0(R_n)^{\times}$.  The homotopy fiber $F_\gamma$ of $(\psi^g - \gamma): R_n \to R_n$ is an invertible spectrum.  When $\gamma = 1$, $F_1$ is equivalent to $S$.

\end{prop}

\begin{proof}

It is not hard to see that $\psi^g - \gamma$ is surjective in $K_*$.  First note that we may regard the $K_*$-Hurewicz image of $\gamma$ as a continuous function $\gamma_* \in K_0(R_n) = C(\Z_p^{\times}, \F_p)$.  If $\phi: \Z_p^{\times} \to \F_p$ is any continuous map, define $f(g^m) \in \F_p$ for $m \in \N$ inductively by the formulae $f(1) = 1$ and $f(g^{m+1}) = \gamma_*(g^m) f(g^m) + \phi(g^m)$.  The function $f$ extends uniquely from the dense subspace $g^\N$ to $\Z_p^{\times}$, since the function $f$ so defined is (uniformly) continuous on $g^\N \leq \Z_p^{\times}$.  Then, since 
$$[\psi^g-\gamma](f)(g^m) = f(g^{m+1})-\gamma_*(g^m)f(g^m) = \phi(g^m),$$
the equality $(\psi^g-\gamma)_*(f) = \phi$ holds on a dense subset; since both are continuous, they must be equal.

Further, the kernel of $(\psi^g-\gamma)_*$ kernel consists of those functions $f: \Z_p^{\times} \to K_*$ which satisfy $f(gx) = \gamma f(x)$.  As such a function is determined by its value on $1$, this has rank one over $K_*$.  Thus $K_*(F_\gamma) = \ker((\psi^g - \gamma)_*)$ is rank one over $K_*$, and hence invertible.

When $\gamma = 1$, the kernel consists of those functions on $\Z_p^{\times}$ which are invariant under translation by $g$, namely the constant functions.  As in Lemma 2.5 of \cite{hms}, the unit of the ring spectrum $R_n$ --  induced by the basepoint inclusion in $K(\Z_p, n+1)$ -- lifts to $F_1$, and carries $K_*(S)$ onto these functions, yielding the desired isomorphism in $K_*$.

\end{proof}

One should think of the second statement in this Proposition as saying that $S$ is the homotopy fixed point spectrum for the action of $\Z_p^{\times}$ on $R_n$ given in Lemma \ref{infty_lem}.

\begin{defn}

Denote by $G$ the invertible spectrum $F_g$ associated to $g \in \Z_p^{\times} \leq \pi_0(R_n)^{\times}$.

\end{defn}

We will show below that $G=F_g$ is in fact the spectrum $\Sdet$ of the introduction.  With Theorem \ref{intro_snaith_thm} in hand, this is largely an exercise in notation; until that result is proven, we must apologise for the multiplicity of different symbols for the same object.

\begin{prop} \label{map_to_pic_prop}

The assignment $\gamma \mapsto F_\gamma$ defines a homomorphism $\pi_0(R_n)^{\times} \to \Pic_n$.  

\end{prop}

\begin{proof}

We must show that $F_{\gamma \sigma} \simeq F_\gamma \otimes F_\sigma$.  Now $\psi^g$ is a map of ring spectra, so this diagram homotopy commutes:
$$\xymatrix{
F_\gamma \otimes F_\sigma \ar[dr]^-{(\gamma \otimes \sigma) \circ i} \ar[d]_-i & \\
R_n \otimes R_n \ar[r]^-{\psi^g \otimes \psi^g} \ar[d]_-\mu & R_n \otimes R_n \ar[d]^-{\mu}\\
R_n \ar[r]^-{\psi^g} & R_n\\
}$$
Since $\mu \circ (\gamma \otimes \sigma ) \circ i = \gamma \sigma \circ i$, we see that $\mu \circ i$ lifts to $F_{\gamma \sigma}$; this map is easily seen to be an isomorphism in $K_*$.

\end{proof}

\begin{cor} \label{lifting_prop}

The canonical map $\delta:G \to X[\alpha^{-1}]$ from the homotopy fibre $G = F_g$ lifts to a map $\rho: G \to X$:
$$\xymatrix{
 & X \ar[r]^-{\psi^g-g} \ar[d]^-j & X \ar[d]^-j \\
G \ar[r]^-\delta \ar[ur]^-{\rho} & X[\alpha^{-1}] \ar[r]^-{\psi^g-g} & X[\alpha^{-1}]
}$$
Furthermore, there are equivalences of ring spectra
$$\xymatrix@1{X[\rho^{-1}] \ar[r]^-\simeq & X[\alpha^{-1}][\delta^{-1}] & X[\alpha^{-1}] \ar[l]_-\simeq}.$$

\end{cor}

\begin{proof} 

We note that Proposition \ref{weird_split_prop} immediately implies the existence of the lift of $\delta$ to $X$.
%lift to one of the terms in the telescope.  That is, there must then be a map $\rho:G \to Z^{-n} \otimes X$ which lifts $\delta$.  Composing with $m_\alpha$ an appropriate number of times if necessary, we may take $n$ to be a multiple of $p-1$, and get the desired map $\rho: G \to X$.

The right equivalence is the localisation map at $\delta$.  By the proof of Proposition \ref{invertible_prop}, the image of a generator for $G$ under $\delta$ is $f_0(x) = x \bmod p$.  As this is clearly invertible in $C(\Z_p^{\times}, K_*) = K_*(X[\alpha^{-1}])$, $m_\delta$ is an equivalence, and thus the directed system defining $X[\alpha^{-1}][\delta^{-1}]$ is constant.

The left equivalence is induced by the map $X \to X[\alpha^{-1}]$ after localisation at $\rho$ (which maps to localisation at $\delta$, since $j \rho = \delta$).  By the same argument as above, the image of a generator for $G$ under $\rho$ is a function $f: \Z_p \to K_*$ which restricts to $f_0$ along the inclusion $\Z_p^{\times} \subseteq \Z_p$.  Replacing $\rho$ with $m_\alpha^{p-1} \cdot \rho$ if necessary, we may assume that $f$ vanishes on $\Z_p \setminus \Z_p^{\times}$, and so $f = f_0$.  Thus the localisation map $X[\rho^{-1}] \to X[\alpha^{-1}][\delta^{-1}]$ is an $K_*$-isomorphism.

\end{proof}

We note that since $\im(\rho_*)$ is generated by the class\footnote{This class is detected by the primitive element $y \in K^*X$.  We thank Mike Hopkins, Jacob Lurie, and Eric Peterson for pointing out that this observation may be promoted to the claim that $G$ can be constructed as $\Sigma \Cotor_X(S, S)$; here $X$ is a coalgebra spectrum, and $\Cotor_X(S, S)$ is the associated (reduced) cobar construction.  See \cite{peterson_annular} for details on this point of view and an algebro-geometric interpretation.} $f_0$, we may conclude that $\dim(G) = \dim(f_0) = 2g(n)$.  We additionally record the following for later use:

\begin{prop}

The spectrum $Z$ is homotopy equivalent to $F_\zeta$, via a map making the following diagram (whose bottom row is a fibre sequence) commute:
$$\xymatrix{
Z \ar[r]^-{\alpha} \ar@{.>}[d] & X \ar[r]^-{\psi^g-\zeta} \ar[d]^-j & X \ar[d]^-j \\
F_\zeta \ar[r] & X[\alpha^{-1}] \ar[r]^-{\psi^g-\zeta} & X[\alpha^{-1}].
}$$

\end{prop}

\begin{proof}

Consider the composite
$$\xymatrix{ L_{K(n)} \Sigma^\infty K(\Z/p, n)_+ \ar[r]^-\pi & L_{K(n)} \Sigma^\infty K(\Z/p, n)_+ \ar[r]^-\beta & X \ar[r]^-{\psi^g} & X}$$
Using the fact that $\psi^g \circ \beta = \beta \circ \psi^\zeta$, we have
\begin{eqnarray*}
\psi^g \circ \beta \circ \pi & = & \beta \circ \psi^\zeta \circ \left(\frac{1}{p-1} \sum_{k=0}^{p-1} \zeta^{-k} (\psi^\zeta)^k\right) \\
 & = & \beta \circ \zeta \cdot \pi \\
 & = & \zeta \cdot \beta \circ \pi.
\end{eqnarray*}
Thus $(\psi^g - \zeta) \circ j \circ \alpha = 0$, and giving rise to the dashed arrow.  The fact that it is a $K_*$-isomorphism (and hence equivalence) follows from inspection of the images of fundamental classes in $K_*(X[\alpha^{-1}])$; they are both generated by $f_0$.

\end{proof}

\subsection{Splitting $R_n$}

There is an analogue for $R_n$ of the splitting of $p$-adic $K$-theory into a wedge of $p-1$ Adams summands:

\begin{prop}

There is an equivalence
$$\bigvee_{k=0}^{p-2} G^{\otimes k} \otimes R_n^{h\mu_{p-1}} \to R_n.$$

\end{prop}

\begin{proof}

There is a natural forgetful map $R_n^{h\mu_{p-1}} \to R_n$.  Additionally, one may produce a map $\vee_{k=0}^{p-2} G^{\otimes k} \to R_n$ by wedging together powers of $\delta$; the product of these maps gives the desired equivalence.  

To see that this map is an isomorphism in $K_*$ (and hence an equivalence), we note that since
$$K_*(R_n) = C(\Z_p^{\times}, \F_{p^n}) = \F_{p^n}[f_0, f_1, f_2, \dots] / ( f_0^{p-1}-1, f_k^p-f_k),$$
then $K_* R_n^{h\mu_{p-1}}$ is the subspace generated by monomials in the $f_i$ whose total degree is a multiple of $p-1$, since $\psi^\zeta(f_k) = \zeta f_k$.  The whole space is a free module over this subalgebra, generated by the classes $\{1, f_0, f_0^2, \dots, f_0^{p-2}\}$, which are the images of $G^{\otimes k}$ under $\delta^k$.

\end{proof}

\subsection{The Morava module of $R_n$} \label{morava_section}

\begin{prop} \label{morava_mod_prop}

There are ring isomorphisms
$$\begin{array}{ccc}
C(\Z_p, {E_n}_*) \to {E_n^{\vee}}_*(X) & and & C(\Z_p^{\times}, {E_n}_*) \to {E_n^{\vee}}_*(R_n)
\end{array}$$
In this description, the action of $\G_n$ on both of these Morava modules is induced by $\detpm: \G_n \to \Z_p^{\times}$ and the natural action of $\Z_p^{\times}$ on $\Z_p$ and $\Z_p^{\times}$.

\end{prop}

\begin{proof}

The proof of the first fact is based on the related computation of ${E_n^\vee}_*{E_n}$ given in \cite{hovey_ops}, section 2.  We recall that $K_n^0(X)$ supports a formal group $F$ isomorphic to $\G_m$ (Prop. \ref{rw_prop}), and note that $E_n^0(X)$ supports a deformation $\overline{F}$ of it to $E_n^0$.  The universal deformation of $\G_m$ over $K_n^0= \F_{p^n}$ is the multiplicative group $\G_m$ over $\W(\F_{p^n})$, and so there is a ring map $\phi: \W(\F_{p^n}) \to E_n^0$ and a $\ast$-isomorphism $f: \overline{F} \to \phi^*(\G_m)$.  Since the formula for multiplication in $\G_m$ involves only the ring elements $0$ and $1$, $\phi^*(\G_m) = \G_m$.  

Thus $\overline{F}$ is isomorphic to the multiplicative group.  Consequently, there is a coordinate $\overline{y} \in E_n^*(X)$ which reduces modulo $\m$ to $y \in K^*(X)$, and presents $\overline{F}$ as the multiplicative group.  Therefore, the action of the Adams operations $\psi^k$ on $X$ gives a continuous, exponential map 
$$\begin{array}{ccc}
a: \Z_p \to E_n^*(X) & \mbox{by} & k \mapsto \psi^k(1+\overline{y}) = 1+[k](\overline{y}) = (1+\overline{y})^k
\end{array}$$
This extends linearly over ${E_n}_*$ to give a continuous ring homomorphism $a: {E_n}_*[\Z_p] \to E_n^*(X)$.  Writing $\mu: X \otimes X \to X$ for the multiplication in $X$, $\mu^*$ is the coproduct on $E_n^*(X)$.  Then since $\psi^k: X \to X$ is a map of ring spectra and $\overline{y}$ is a coordinate on the multiplicative group,
\begin{eqnarray*}
\mu^* (\psi^k (1+\overline{y})) & = & (\psi^k \otimes \psi^k)(\mu^*(1+\overline{y})) \\
 & = & (\psi^k \otimes \psi^k)(1\otimes 1+\overline{y} \otimes 1 + 1\otimes \overline{y} + \overline{y} \otimes \overline{y}) \\
 & = & \psi^k(1+\overline{y}) \otimes \psi^k(1+\overline{y}).
\end{eqnarray*}

Therefore $a$ is a map of Hopf algebras.  Applying the functor $\Hom^c_{{E_n}_*}(-, {E_n}_*)$ of continuous ${E_n}_*$-module homomorphisms into ${E_n}_*$ yields a Hopf algebra map
$$a^*: \Hom^c_{{E_n}_*}(E_n^*(X), {E_n}_*) \to \Hom^c_{{E_n}_*}({E_n}_*[\Z_p], {E_n}_*) \cong C(\Z_p, {E_n}_*).$$
The identification of the codomain uses the fact that ${E_n}_*[\Z_p]$ is free.  Similarly, the domain is %
$$\Hom^c_{{E_n}_*}(E_n^*(X), \varprojlim_I (E_n / I)_*) = \varprojlim_I \Hom((E_n / I)^*(X), (E_n / I)_*) = \varprojlim_I  (E_n / I)_*(X) ={ E_n^{\vee}}_*(X)$$
where the limit is taken over the ideals $I = (p^{i_0}, u_1^{i_1}, \dots, u_{n-1}^{i_{n-1}})$.  We know that ${E_n^{\vee}}_*(X)$ is pro-free, concentrated in even dimensions, with reduction modulo $\m$ isomorphic to $K_*(X) = C(\Z_p, \F_{p^n}[u^{\pm 1}])$.  The ring $C(\Z_p, {E_n}_*)$ is also pro-free, and $a^*: E_n^{\vee}(X) \to C(\Z_p, {E_n}_*)$  reduces modulo $\m$ to the isomorphism $\phi^*: (E_n / \m)_*(X) \to C(\Z_p, \F_{p^n}[u^{\pm 1}])$ of Prop. \ref{Z_p_prop}; it is then an isomorphism itself.  

The same localisation technique as in the proof of Cor. \ref{unit_cor} yields the corresponding result for $R_n$.  Specifically, the image of the fundamental class of $Z$ under $\alpha$ is a function $B_0: \Z_p \to {E_n}_*$ whose reduction modulo $\m$ is $b_0 = \xi f_0 u^{g(n)}$.  As in the residue field, an element $m \in \Z_p$ is invertible if and only if $B_0(m)$ is.

To see that the $\G_n$ action is as claimed, we employ Peterson's adaptation \cite{peterson} of Ravenel-Wilson's results \cite{rw} to identify the formal spectrum $\Spf E_n^* K(\Z_p, n+1)$ in terms of the $n^{\rm th}$ exterior power of the $p$-divisible group associated to $E_n$.  Concretely, we recall that the action of $\bS_n$ on ${E_n^\vee}_*(K(\Z_p, 2)) \cong {E_n^\vee}_*(K(\mup, 1))$ is via the defining action of $\OO_n$ on $\Gamma_n$.  Furthermore the Hopf ring circle product satisfies $a \circ b = - b \circ a$.  Therefore, the action of $\bS_n$ on classes in ${E_n^\vee}_*(K(\Z_p, n+1)) \cong {E_n^\vee}_*(K(\mup, n))$ lying in the image of the iterated circle product
$${E_n^\vee}_*(K(\mup, 1)^{\times n}) \to {E_n^\vee}_*(K(\mup, n))$$
is via the $n^{\rm th}$ exterior power of the defining action.  All classes in $K_*(K(\mup, n)))$ may be obtained this way, and the above analysis lifts this statement to ${E_n^\vee}_*$.  Therefore the action of $\bS_n$ must be via $\det$.  Examining the action of elements of $\Gal(\F_{p^n}/\F_p)$, we note that if $\sigma$ is the Frobenius homomorphism, then  
$$\sigma \cdot y = \sigma \cdot(\xi x u^{g(n)}) = \xi^p x u^{g(n)} = \xi^{p-1} y = (-1)^{n-1} y,$$
Here, $x u^{g(n)}$ is fixed by the action of Frobenius, since it descends to a class in $K(n)^*(X)$.  This yields the result.

\end{proof}

\begin{thm} \label{e_infty_maps_thm}

The space of $E_\infty$ maps $R_n \to E_n$ is an infinite loop space with contractible components, the set of which is isomorphic to $\Z_p^{\times}$.  Furthermore, the action of $\G_n \simeq \Aut_{E_\infty}(E_n)$ on this set is via the homomorphism $\detpm: \G_n \to \Z_p^{\times}.$

\end{thm}

\begin{proof}

The proof is essentially the same as the main technical result of \cite{sati_wes}, so we will be brief.  The Goerss-Hopkins-Miller obstruction machinery shows that the higher homotopy groups of $\Map_{E_\infty}(R_n, E_n)$ vanish if the cotangent complex for the map $\F_{p^n} \to (E_n^\vee)_0(R_n)/ \m$ is contractible.  But the latter is
$$(E_n^\vee)_0(R_n)/ \m \cong C(\Z_p^{\times}, \F_{p^n}) \cong \F_{p^n}[f_0, f_1, f_2, \dots] / ( f_0^{p-1}-1, f_k^p-f_k).$$
The cotangent complex is contractible because the Frobenius on this ring is evidently an isomorphism.

The set of components of $\Map_{E_\infty}(R_n, E_n)$ is then the set of continuous ${E_n}_*$-algebra homomorphisms $\Hom^c_{{E_n}_*-alg}({E_n^\vee}_* R_n, {E_n}_*)$, which may be identified with $\Z_p^{\times} = \G_n / \SGn$ by the previous result. 

\end{proof}

One may phrase this in more invariant terms: ${E_n^\vee}_* R_n$ may be identified as $C(\Isom(F, \G_m), {E_n}_*)$, in a variant on the comments after the proof of Proposition \ref{Z_p_prop}.  Then the set of components of the $E_\infty$ mapping space is $\Isom(F, \G_m)$.  This is evidently a torsor for the group $\Z_p^{\times} = \Aut(\G_m)$.

Reduction of a chosen generator of the group of such maps $R_n \to E_n$ modulo $\m$ gives a natural transformation from $R_n$ to $K$ which will be useful in the following sections.

\begin{prop} \label{t_prop}

The cohomology class $1+y = \varphi \bmod \m \in K^*(X) = [X, K]$ extends over $R_n$ to give a map of $A_\infty$-ring spectra $t: R_n \to K$.

\end{prop}

\begin{proof}

It was shown in \cite{sati_wes} that the homotopy class $1+y$ contains an $A_\infty$ representative using Hochschild cohomology methods.  Thus it suffices to show that $(1+y) \circ \alpha: Z \to K$ is invertible.  This follows, since the image of the fundamental class of $Z$ under $\alpha$ is $b_0$, and $\langle 1+y, b_0 \rangle = \xi$.

\end{proof}

\subsection{$R_n$ as a homotopy fixed point spectrum} \label{fixed_section}

We now lift the results of Theorem \ref{e_infty_maps_thm} to an equivalence between $R_n$ and $E_n^{h\SGn}$.  Some setup is required to study such homotopy fixed point spectra.  Let $H$ be a closed subgroup of $\G_n$, and let $A$ be a $K(n)$-local $E_\infty$ ring spectrum, with the property that its Morava module is isomorphic to the ring of continuous functions on $\G_n / H$:
$${E_n^\vee}_*A \cong C(\G_n / H, {E_n}_*).$$
Let $K \leq \G_n$ be another closed subgroup.  A consequence of \cite{dev_hop} is that the homotopy fixed point spectrum $E_n^{hK}$ admits a model which is an $E_\infty$-ring spectrum.  When $K = U$ is open, Devinatz-Hopkins construct a fibrant cosimplicial $E_\infty$-algebra whose totalisation is the spectrum $E_n^{hU}$.  This cosimplicial spectrum is a rectification of an $h_\infty$ cosimplicial $E_n$-Adams resolution of $E_n^{hU}$, were it to exist (see also \cite{bd}).  When $K$ is closed, the homotopy fixed point spectrum is constructed as a colimit over open subgroups $U$ of $\G_n$ containing $K$: $E_n^{hK} = L_{K(n)}( \colim_U E_n^{hU})$.

Following their lead, but using their result that for any closed $K$, $E_n^{hK}$ \emph{does} exist, we will consider an actual cosimplicial $K(n)$-local $E_n$-based Adams resolution of $E_n^{hK}$.  Define 
$$B_K^s = E_n^{\otimes s+1} \otimes E_n^{hK}  ; \; s \geq -1.$$
Taking $s\geq 0$, one may equip $B_K^\bullet$ with the structure of a cosimplicial spectrum in a familiar fashion, inserting units for coface maps, and applying multiplication for codegeneracies.  We may freely replace $B_K^\bullet$ with a fibrant cosimplicial spectrum without changing the homotopy type of the terms $B_K^s$.  The natural coaugmentation $E_n^{hK} = B_K^{-1} \to \Tot(B_K^\bullet)$ is a weak equivalence, as can be seen from the associated Adams spectral sequence, and Theorem 2.(ii) of \cite{dev_hop}.  Lastly $\Tot(B_K^\bullet)$, being the totalisation of a cosimplical $E_\infty$ algebra, is itself $E_\infty$; the coaugmentation is an equivalence of $E_\infty$ algebras. 

\begin{thm} \label{awesome_lemma}

Let $K$ be a closed subgroup of $\G_n$. The space of maps of $E_\infty$ ring spectra from $A$ to $E_n^{hK}$ has contractible components.  Furthermore, there is a bijection from $\pi_0(\Map_{E_{\infty}}(A, E_n^{hK}))$ to the set 
$$\{ \mbox{$xH \in \G_n/H$ which conjugate $K$ into $H$} \} \subseteq \G_n / H;$$ 
Lastly, if $K$ is conjugate to $H$, each of these maps is an equivalence. 

\end{thm}

\begin{proof}

The Morava module of $B_K^s$ may be identified as the ring $C(\G_n^{\times s+1} \times (\G_n / K), {E_n}_*)$.  By the same sort of argument as in the previous section, we see that the cotangent complex for the reduction of the degree $0$ part of this algebra by $\m$ (that is, $C(\G_n^{\times s+1} \times (\G_n / K), \F_{p^n})$) is contractible.   Thus the space $\Map_{E_\infty}(A, B_K^s)$ of maps of $E_\infty$ ring spectra from $A$ to $B_K^s$ has contractible components which are in bijection with the set $\Map(\G_n^{\times s+1} \times (\G_n / K), (\G_n / H))^{\G_n}$ of continuous maps, equivariant for the action of $\G_n$ (diagonal on the domain):
\begin{eqnarray*} 
\pi_0 \Map_{E_\infty}(A, B_K^s) & \cong & \Hom^c_{{E_n}_*-alg/{E_n^\vee}_* {E_n}}({E_n^\vee}_*(A), {E_n^\vee}_*(B^s_K)) \\
 & = & \Map^\delta(\G_n^{\times s+1} \times (\G_n / K), (\G_n / H))^{\G_n} %\\
% & = & \Map^\delta(\G_n^{\times s+1} \times_{\G_n} (\G_n / K), (\G_n / H)) 
\end{eqnarray*}
Here, we are equipping this function space with the discrete topology; this is the meaning of the superscript $\delta$.  The discrete topology on $\Map(\G_n^{\times s+1} \times (\G_n / K), (\G_n / H))^{\G_n}$ is \emph{not} homeomorphic to the natural (compactly generated compact-open) one, which we will denote with a superscript $co$.  However, since $\G_n / H$ is totally disconnected, we note that for all $s$, the evaluation at any point in $\Delta^s$ gives a bijection
$$\Map(\Delta^s,  \Map^{co}(\G_n^{\times s+1} \times (\G_n / K), (\G_n / H))^{\G_n}) =  \Map^{co}(\G_n^{\times s+1} \times (\G_n / K), (\G_n / H))^{\G_n}. \leqno{(*)}$$
Additionally, since $\Map^\delta$ is equipped with the discrete topology, the same evaluation yields a bijection
$$\Map(\Delta^s,  \Map^{\delta}(\G_n^{\times s+1} \times (\G_n / K), (\G_n / H))^{\G_n}) =  \Map^{\delta}(\G_n^{\times s+1} \times (\G_n / K), (\G_n / H))^{\G_n}.  \leqno{(**)}$$

We note that $\G_n^{\times s+1} \times (\G_n / K)$ is the $s^{\rm th}$ term of a simplicial $\G_n$-space, the product ${E\G_n}_\bullet \times \G_n / K = \G_n^{\times \bullet+1} \times \G_n / K$.  Therefore $\Map^{co}(\G_n^{\times s+1} \times (\G_n / K), (\G_n / H))^{\G_n}$ is the $s^{\rm th}$ term of the cosimplicial space of $\G_n$-equivariant, continuous functions
$$\Map^{co}({E\G_n}_\bullet \times \G_n / K, \G_n / H)^{\G_n}$$
Being a function object from a simplicial space, this is a fibrant cosimplicial space.  

In contrast, $\Map^\delta(\G_n^{\times s+1} \times (\G_n / K), (\G_n / H))^{\G_n}$ is the $s^{\rm th}$ term of the discretisation of this cosimplicial space, $\Map^{\delta}({E\G_n}_\bullet \times \G_n / K, \G_n / H)^{\G_n}$.  It is also fibrant, being a cosimplicial \emph{set}.  The identity is a continuous map
$$\Tot(\Map^{\delta}({E\G_n}_\bullet \times \G_n / K, \G_n / H)^{\G_n}) \to \Tot(\Map^{co}({E\G_n}_\bullet \times \G_n / K, \G_n / H)^{\G_n}).$$
We have no expectation that this map is a homeomorphism.  However, equations $(*)$ and $(**)$ imply that it \emph{is} a bijection.  In fact, this totalisation can be computed as follows:
\begin{eqnarray*}
\Tot(\Map^{co}({E\G_n}_\bullet \times \G_n / K, \G_n / H)^{\G_n}) & \cong & \Map^{co}(E \G_n \times (\G_n / K), \G_n / H)^{\G_n} \\
 & \cong & \Map^{co}(\G_n / K, \G_n / H)^{\G_n} %\\
% & \cong & (\G_n / H)^K.
\end{eqnarray*} 
Here $\cong$ indicates homeomorphism.  The last homeomorphism is induced by restriction to the basepoint in $E\G_n$; it is a homeomorphism since $E\G_n$ is connected and $\G_n / H$ is totally disconnected.  Note that an element $f \in \Map^{co}(\G_n / K, \G_n / H)^{\G_n}$ is specified by the element $xH = f(K) \in \G_n / H$.  It must be the case that for each $k \in K$, 
$$xH = f(K) = f(kK) = kf(K) = kxH$$
and so $K \subseteq xHx^{-1}$.  Thus this set is empty if $K$ is not subconjugate to $H$, and in fact the set of all such maps $f$, being determined by $x$, is given precisely by these cosets:
$$\Tot(\Map^{\delta}({E\G_n}_\bullet \times \G_n / K, \G_n / H)^{\G_n}) = \{ \mbox{$xH \in \G_n/H$ which conjugate $K$ into $H$} \}.$$

%We may conclude that evaluation at the basepoint is a bijection
%%
%$$\Tot(\Map^{\delta}({E\G_n}_\bullet \times_{\G_n} \G_n / U, \G_n / H)) = (\G_n / H)^U.$$
%%
%If $U$ is not a subgroup of $H$, there is an element $u \in U$ that acts nontrivially on $\G_n / H$, and so the set of fixed points is empty.  Conversely, if $U \leq H$, the action is trivial.  
%

Now, since $E_n^{hK} \simeq \Tot(B_K^\bullet)$ is (equivalent to) the totalisation of a cosimplicial $E_\infty$-algebra, we have 
$$\Map_{E_\infty}(A, E_n^{hK}) \simeq \Tot(\Map_{E_\infty}(A, B_K^\bullet)).$$
Taking $B_K^\bullet$ to be a fibrant cosimplicial $E_\infty$-algebra, $\Map_{E_\infty}(A, B_K^\bullet)$ is a fibrant cosimplicial space. %reference: Definition VII.4.2 of EKMM ==>   fibration of latching objects on spectrum level gives fibration on space level.
Now, the projection to the set of components, 
$$\Map_{E_\infty}(A, B_K^\bullet) \to \pi_0 \Map_{E_\infty}(A, B_K^\bullet) = \Map^{\delta}({E\G_n}_\bullet \times \G_n / K, \G_n / H)^{\G_n}$$
is a levelwise equivalence of fibrant cosimplicial spaces; the fibrancy of domain and codomain yields an equivalence of totalisations.  Therefore $\Map_{E_\infty}(A, E_n^{hK})$ is empty if $K$ is not subconjugate $H$ and has components in bijection with the set of $xH \in \G_n / H$ which conjugate $K$ into $H$.

Finally, if $K = xHx^{-1}$, the Morava modules of both $A$ and $E_n^{hK}$ are both isomorphic to $C(\G_n / K, {E_n}_*)$, so the element of the $E_\infty$ mapping space corresponding to multiplication by $x$ implements an isomorphism of Morava modules, and hence an equivalence.

\end{proof}

\begin{cor} \label{snaith_cor}

The space of weak equivalence of $E_\infty$ ring spectra $R_n \to E_n^{h\SGn}$ has contractible components, and the set of which is a torsor for $\Z_p^{\times}$, where $k \in \Z_p^{\times}$ acts by post-composition with the Adams operation $\psi^k$.

\end{cor}

Recall that $\Z_p^\times$ acts up to homotopy on $R_n$ by way of a strict action on $L_{K(n)} \Sigma^\infty K(\Z_p, n+1)_+$.  Any equivalences $e$ constructed in the Corollary is equivariant (up to homotopy) for the $\Z_p^{\times}$-action.  This can be seen as follows.  For any $k \in \Z_p^{\times}$, the action of $\psi^k$ on $R_n$ is an $E_\infty$-map, via Lemma \ref{infty_lem}.  Comparing the results of Theorem \ref{awesome_lemma} (with $H= \SGn$) for $K=\{1\}$ and $\SGn$, we see that the forgetful map
$$\Map_{E_\infty}(R_n, E_n^{h\SGn}) \to \Map_{E_\infty}(R_n, E_n)$$
is a homotopy equivalence (and both spaces have contractible components in bijection with $\Z_p^{\times}$).  So to see that $\psi^k \circ e \simeq e \circ \psi^k$, it suffices to show that their composite with the forgetful map into $E_n$ are homotopic.  But this is follows immediately from Proposition \ref{morava_mod_prop}.

\subsection{A variant on $R_n$ when $n$ is even} \label{old_R_n_section}

Consider the homomorphism $\det: \G_n \to \Z_p^\times$ which is given by the determinant on $\bS_n$, and is trivial on $\Gal(\F_{p^n}/\F_p)$; let $S \G_n = \ker(\det)$.  This group differs slightly from $\SGn$ when $n$ is even; we will assume this throughout this section.  Let $H = S \G_n \cap \SGn$; this is index $2$ in both of these groups and normal in $\G_n$.  Further, $S\G_n / H$ is generated by $F \in \Gal(\F_{p^n}/\F_p)$, subject to the relation $F^2 = 1$.

\begin{prop}

The spectrum $E_n^{hH}$ splits as the wedge $E_n^{hS\G_n} \vee E_n^{h\SGn}$.

\end{prop}

\begin{proof}

Since $p$ is odd, we define an idempotent $q: E_n^{hH} \to E_n^{hH}$ as in Proposition \ref{small_split_prop} by $q = \frac{1}{2}(1-F)$, and define a spectrum $Q$ which represents the image of this functor inside of $E_n^{hH}$; then $E_n^{hH}$ splits as a wedge $Q \vee Q^{\perp}$.  Note that $Q^\perp$ represents the image of the complementary idempotent $1-q = \frac{1}{2}(1+F)$.  Further, one can see that the composite of the forgetful map and projection
$$E_n^{hS\G_n} \to E_n^{hH} \to Q^\perp$$
is an equivalence by examining the induced maps in $K_*$: $K_*(E_n^{hH}) \cong C(\G_n / H, K_*)$, and the (injective) image of $K_*(E_n^{hS\G_n})$ consists of those functions which factor through $\G_n / S\G_n$.  This is precisely $K_*(Q^\perp) = \ker(q_*) = \{f: \G_n / H \to K_* \; | \; f(Fg) = f(g)\}$.  A similar argument shows that $Q \simeq E_n^{h\SGn}$.

\end{proof}

Recall the primitive $2p-2^{\rm nd}$ root of unity $\xi$, which we will regard as an element of $\pi_0 E_n$.  Note that it is invariant under $H \leq \G_n$ (the relevant fact being that $F(\xi) = -\xi$, so $\xi$ is fixed by $\Gal(\F_{p^n}/\F_p) \cap H = \langle F^2 | F^n \rangle$).  Thus, multiplication by $\xi$ defines an automorphism $\xi: E_n^{hH} \to E_n^{hH}$.

Note that both $E_n^{hS\G_n}$ and $E_n^{h\SGn}$ admit residual actions of $\Z_p^{\times}$ as quotients of $\G_n$.  

\begin{prop}

The map $\xi: E_n^{hH} \to E_n^{hH}$ restricts to equivalences
$$\xi: E_n^{hS\G_n} \to E_n^{h\SGn} \; \mbox{ and } \; \xi: E_n^{h\SGn} \to E_n^{hS\G_n}$$
These are not maps of ring spectra, although they are $\Z_p^{\times}$-equivariant, up to homotopy.

\end{prop}

\begin{proof}

We note that $q \circ \xi = \frac{1}{2}(\xi+\xi F) = \xi \circ (1-q)$.  That is, the automorphism $\xi$ interchanges the two complementary idempotents, and hence the wedge factors $Q$ and $Q^\perp$.  %This yields the first claim.  

There are two natural homomorphisms $\G_n / H \to \Z_p^{\times}$ given by $\det$ and $\detpm$, respectively.  Indeed, $\det$ and the projection onto the Galois factor give an isomorphism
$$\xymatrix{\G_n / H \ar[rr]^-{\det \times \proj} & & \Z_p^{\times} \times \Gal(\F_{p^n}/\F_p) / (H \cap  \Gal(\F_{p^n}/\F_p)) = \Z_p^{\times} \times \langle F | F^2 \rangle.}$$
With respect to this splitting, $\det$ is the projection onto the first factor, and $\detpm$ is that projection, convolved with the sign character of the second factor.

Let $g \in \Z_p^{\times}$ be a topological generator and let $h \in \G_n / H$ be the element with $\det(h) = g$ and $\proj(h) = 1$; then also $\detpm(h) = g$.  Denote by $\chi^h$ the action of $h$ on $E_n^{hH}$, and by $\varphi^g$ the action of $g$ on $E_n^{hS\G_n}$; then this diagram commutes up to homotopy:
$$\xymatrix{
E_n^{hS\G_n} \ar[r] \ar[d]_-{\varphi^g} & E_n^{hH} \ar[d]_-{\chi^h} & E_n^{h\SGn} \ar[d]^-{\psi^g} \ar[l] \\
E_n^{hS\G_n} \ar[r] & E_n^{hH} & E_n^{h\SGn}. \ar[l]
}$$
Since $\proj(h) = 1$, $\xi$ commutes with $\chi^h$.  Since $\xi$ permutes the two wedge factors $Q$ and $Q^\perp$, it therefore throws $\psi^g$ onto $\varphi^g$, and vice versa.

\end{proof}

\subsection{Gross-Hopkins duality}

The previous sections imply that there is a homotopy commutative diagram
$$\xymatrix{
R_n \ar[d]_-{\psi^g - g} \ar[r]^-{\simeq} &  E_n^{h\SGn} \ar[d]_-{\psi^g - g} \ar[r]_-{\xi}^-{\simeq}  & E_n^{hS\G_n} \ar[d]^-{\varphi^g - g} \\
R_n \ar[r]^-{\simeq}  & E_n^{h\SGn} \ar[r]_-{\xi}^-{\simeq} & E_n^{hS\G_n} 
}$$
where the action of $\psi^g \in \Z_p^{\times}$ on $E_n^{h\SGn}$ is from the residual $ \G_n / \SGn$ Morava stabiliser action (and the second square is only required when $n$ is even).  Consequently the homotopy fibres of each column are equivalent.  In \cite{ghmr_pic}, the homotopy fibre of the rightmost arrow was named $\Sdetold$, and its Morava module was shown to be ${E_n^{\vee}}_*(\Sdetold) = {E_n}_*\langle \det \rangle$. 

In \cite{gross-hopkins, strickland_gh} it was shown that this is the same Morava module as for $\Sigma^{n-n^2} I_n$, where $I_n$ is the Brown-Comenetz dual of $M_n S^0$, the $n^{\rm th}$ monochromatic layer of the sphere spectrum.  It was also shown that there and in \cite{hms} that for $2p-2 \geq \max \{n^2, 2n+2 \}$, an invertible spectrum is determined by its Morava module.  We conclude:

\begin{cor}

There is an equivalence $G \simeq \Sdetold$.  When $2p-2 \geq \max \{n^2, 2n+2 \}$, these may also be identified with $\Sigma^{n-n^2} I_n$. 

\end{cor}

\section{Thom spectra and characteristic classes}

\subsection{Orientations and cannibalistic classes for $R_n$}

We briefly review the theory of \cite{may_e_infty, abghr} for orientations of Thom spectra, especially from the point of view of \cite{may_good_for}.  Recall that for an $E_\infty$ ring spectrum $R$, $\gl_1 R$ is the spectrum whose infinite loop space is $\GL_1 R$, the space of units in $\Omega^\infty R$.  Further, $\gl_1$ is functorial for maps of $E_\infty$-ring spectra.

Let $\eta: S \to R_n$ be the unit of $R_n$, and define $b(S, R_n)$ to be the cofibre of $\gl_1 \eta: \gl_1 S \to \gl_1 R_n$.  Write $B(S, R_n) = \Omega^\infty b(S, R_n)$; then there is an equivalence 
$$B(S, R_n) \simeq B(*, \GL_1 S, \GL_1 R_n)$$
between $B(S, R_n)$ and the bar construction for the action of $\GL_1 S$ on $\GL_1 R_n$ via $\GL_1(\eta)$.  The fibre sequence of infinite loop spaces associated to this construction is of the form
$$\xymatrix{
\GL_1 S \ar[r]^-{\eta} & \GL_1 R_n \ar[r] & B(S, R_n) \ar[r]^-{\gamma} & B\GL_1 S
}$$

We recall that (homotopy classes of) maps of spaces $\zeta:Y \to B\GL_1 S$ define Thom spectra with ``fibre" $S$ -- for such a map $\zeta$, the Thom spectrum $Y^\zeta$ is defined as 
$$Y^\zeta = \Sigma^\infty P_{+} \otimes_{\Sigma^\infty \GL_1 S_+} S,$$ 
where $P$ is the homotopy fibre of $\zeta$.  Furthermore (see section 3 of \cite{may_good_for}), the set of lifts of $\zeta$ over $\gamma$ to elements of $[Y, B(S, R_n)]$ are in bijection with $R_n$-orientations of $\zeta$, that is, Thom classes $u: Y^\zeta \to R$.

Consider a space $Y$ and a based map $f: Y \to B(S, R_n)$; write $\zeta = \gamma \circ f$.  By the above, $f$ consists of the data of an orientation $u: Y^\zeta \to R_n$.  Let $A \in \Pic_n$, and note that a Thom class for $A\zeta$ is given by
$$1 \otimes u: A \otimes Y^\zeta = Y^{A\zeta} \to A \otimes R_n$$
If we insist that $A$ is a power of $G$, we may make employ the periodicity of $R_n$ to make the following definition:

\begin{defn}

The \emph{normalised} Thom class $\ubar = \ubar(G^m \zeta) := \delta^m u \in R^0(Y)$ is defined as the composite
$$\xymatrix{
Y^{G^m \zeta} = G^m \otimes Y^\zeta \ar[r]^-{1 \otimes u} & G^m \otimes R_n \ar[r]^-{\delta^m \otimes 1} & R_n \otimes R_n \ar[r]^-\mu & R_n
}$$
Then for any $k \in \Z_p^{\times}$, define $\theta_k(G^m \zeta) \in R_n^0(Y)$ by the equation 
$$\psi^k(\ubar) = \theta_k(G^m \zeta) \cdot \ubar.$$
Following Adams, we call $\theta_k(G^m \zeta)$ the (Bott) \emph{cannibalistic class} of $G^m\zeta$.

\end{defn}

\begin{prop} \label{easy_prop}

Let $\zeta$ and $\xi$ be Thom spectra associated to classes $f, h: Y \to B(S, R_n)$.  Then

\begin{enumerate} 

\item $\theta_k(0) = 1$.

\item $\theta_k(\zeta + \xi) = \theta_k(\zeta) \cdot \theta_k(\xi)$.

\item $\theta_k(G^m\zeta) = k^m \theta_k(\zeta)$.

\item $\theta_{kl}(\zeta) = \psi^k(\theta_l(\zeta)) \cdot \theta_k(\zeta)$.

\end{enumerate}

\end{prop} 

\begin{proof}

The proof that $\theta_k$ is exponential depends upon the fact that a product of Thom classes for $\zeta$ and $\xi$ is a Thom class for $\zeta + \xi$.  Note that this implies that $\theta_k(\zeta)^{-1}$ exists, and is equal to $\theta_k(-\zeta)$.  The third property uses the fact that $\psi^k(\delta^m \otimes u) = k^m \delta^m \otimes \psi^k(u)$.  The last is an application of the equation $\psi^{kl} = \psi^k \psi^l$.

\end{proof}

We note that this proposition implies that $\theta_k(\zeta)$ is invertible; $\theta_k(\zeta) \in R_n^0(Y)^{\times}$.  Further, the last part of this Proposition allows us to regard the collection $\{\theta_k(\zeta); \; k \in \Z_p^{\times}\}$ of all of the cannibalistic classes of $\zeta$ as a 1-cocycle $\theta(\zeta): \Z_p^{\times} \to R_n^0(Y)^\times$, where the action of $\Z_p^{\times}$ on $R_n^0(Y)^\times$ is through Adams operations.

Now for all $k$, $\psi^k: R_n \to R_n$ is a ring map, so there is a well defined map $\gl_1 \psi^k: \gl_1 R_n \to \gl_1 R_n$.  Write $\psi^k/1: \gl_1 R_n \to \gl_1 R_n$ for the difference (using the \emph{multiplicative} infinite loop space structure on $\GL_1 R_n$) between this map and the identity.  Since $\psi^k \circ \eta = \eta$, we see that $(\psi^k/1) \circ \gl_1 \eta$ is nullhomotopic.   Therefore there is a natural map $c(\psi^k): b(S, R_n) \to \gl_1 R_n$ making this diagram commute
$$\xymatrix{
\gl_1 S \ar[r]^-{\gl_1 \eta} & \gl_1 R_n \ar[dr]_-{\psi^k/1} \ar[r] & b(S, R_n) \ar[r]^-{\gamma} \ar[d]^-{c(\psi^k)} & \Sigma \gl_1 S  \\
 & & \gl_1 R_n &
}$$
Then the cannibalistic classes  may be computed via the operation $c(\psi^k)$:
$$\theta_k(\zeta) = c(\psi^k) \circ f;$$
see, e.g., Proposition 3.5 of \cite{may_good_for}.

\subsection{A model for $B(S, R_n)$}

Let $r$ denote a topological generator for the $p$-adic units $\Z_p^{\times}$.  At times it will be useful to take $r$ to be $g = \zeta(1+p)$, where $\zeta$ is a primitive $p-1^{\rm st}$ root of unity.  At others it will be beneficial to choose $r \in \N$; e.g., any positive integer whose reduction modulo $p^2$ generates $(\Z / p^2)^{\times}$.  We recall from Proposition \ref{invertible_prop} that for $r=g$ there is a cofibre sequence
$$\xymatrix{
S \ar[r]^-{\eta} & R_n \ar[r]^-{\psi^r-1} & R_n
}$$
This may be shown to be a fibre sequence for other generators $r$ using the fact that $\Z_p^{\times}$ is topologically cyclic.  Alternatively, this follows from the $\Z_p^{\times}$-equivariant identification $R_n \simeq E_n^{h\SGn}$ of Corollary \ref{snaith_cor}.

\begin{prop}

There is an equivalence between $\Omega^\infty S = \Omega^\infty \hofib(\psi^r-1:R_n \to R_n)$ and the homotopy equaliser of the maps $\psi^r$ and $1: \Omega^\infty R_n \to \Omega^\infty R_n$, 
$$\hoeq((\psi^r, 1): \Omega^\infty R_n \to \Omega^\infty R_n).$$
Similarly, there is an equivalence between $\hofib(\psi^r/1: \GL_1 R_n \to \GL_1 R_n)$ and the homotopy equaliser of the maps $\psi^r$ and $1: \GL_1 R_n \to \GL_1 R_n$, 
$$\hoeq((\psi^r, 1): \GL_1 R_n \to \GL_1 R_n).$$

\end{prop}

\begin{proof}

In the first case, an element of the homotopy fibre consists of a pair $(x, f)$, where $x \in \Omega^\infty R_n$, and $f$ is a path in $\Omega^\infty R_n$, starting at the basepoint (the additive unit $0$), and ending at $\psi^r(x)-x$.  An element of the homotopy equaliser is a path $h$ in $\Omega^\infty R_n$, starting at a point $x$, and ending at $\psi^r(x)$.  An equivalence between these spaces is gotten by sending $(x, f)$ to the path $h$ gotten from $f$ by pointwise addition of $x$.

The proof of the second claim is the same as for the first, replacing addition with multiplication and $0$ with $1$.  
\end{proof}

%\begin{prop}
%
%There is an equivalence between $\hofib(\psi^r/1: \GL_1 R_n \to \GL_1 R_n)$ and the homotopy equaliser of the maps $\psi^r$ and $1: \GL_1 R_n \to \GL_1 R_n$, 
%%
%$$\hoeq((\psi^r, 1): \GL_1 R_n \to \GL_1 R_n).$$
%
%\end{prop}

\begin{prop} 

$\GL_1 S$ is the union of the components of $\Omega^\infty S$ lying over $\GL_1 R_n$.

\end{prop}

\begin{proof}

Equivalently, we must show that for every $x \in \pi_0(S)$, if its image $\eta_*(x) \in \pi_0(R_n)$ is invertible, then so too is $x$.  Let $y = \eta_*(x)^{-1}$; since $\psi^r$ is a ring homomophism,
$$1 = \psi^r(1) = \psi^r(\eta_*(x)) \psi^r(y) = \eta_*(x) \psi^r(y)$$
and so $\psi^r(y) = y$.  Thus $y$ lifts to a class $y' \in \pi_0(S)$ with $\eta_*(y') = y$, and so $xy' = 1 + w$, where $w \in \ker(\eta_*) = \im [\partial_*: \pi_1 R_n \to \pi_0(S)]$.  We note that $w^2 = 0$: $w$ is represented by an element in $E_2^{1, 1}$ of the homotopy fixed point spectral sequence $H^*(\Z_p^{\times}, \pi_*(R_n)) \implies \pi_*(S)$, and so its square is represented by a class in $E_2^{2, 2}$ (modulo terms of \emph{higher} filtration).  This is $0$, since the homological dimension of $\Z_p^{\times}$ is $1$.  Therefore $y'(1-w)$ is a multiplicative inverse for $x$.

% But $\eta = j \circ \eta_X$, where $\eta_X: S \to X$ is the unit.  Since $0 = j_* (\eta_X)_*(w)$, Proposition \ref{weird_split_prop} provides a class $w' \in \pi_0 (X)$ with $\psi^p(w') = (\eta_X)_*(w)$.  \marginpar{fix this}
%
%$$\ker(\eta_*) = (\eta_X)_*^{-1}(\ker j_*) = (\eta_X)_*^{-1}(\psi^p_* \pi_0(X))$$     However, $\eta_X$ is split, with one-sided inverse induced by the projection $K(\Z_p, n) \to *$

\end{proof}

\begin{cor}

For any topological generator $r$, the infinite loop map $c(\psi^r): B(S, R_n) \to \GL_1 R_n$ is an equivalence on connected components.

\end{cor}

\begin{proof}

The previous two Propositions allow us to conclude that there is an equivalence $\GL_1 S \simeq \hofib(\psi^r/1: \GL_1 R_n \to \GL_1 R_n)$.  This allows us to identify the lower fibre sequence in the following commuting diagram of fibrations:
$$\xymatrix{
\Omega B(S, R_n) \ar[r]^-\gamma \ar[d]_-{c(\psi^r)} & \GL_1 S \ar[r]^\eta \ar@{=}[d] & \GL_1 R_n \ar[r] \ar@{=}[d] & B(S, R_n) \ar[d]^-{c(\psi^r)} \\
\Omega \GL_1 R_n \ar[r]_-\gamma & \GL_1 S \ar[r]_\eta  & \GL_1 R_n \ar[r]_{\psi^r/1} & \GL_1 R_n  \\
}$$
The leftmost column is an equivalence, since both spaces are homotopy fibres of the same map; this gives the result.

\end{proof}

Perhaps surprisingly, this allows us to conclude that an $R_n$-orientation of a Thom spectrum $Y^\zeta$ over a connected space $Y$ is uniquely determined by its $r^{\rm th}$ cannibalistic class $\theta_r(\zeta)$ (or equivalently, the cocycle $\theta(\zeta)$).  More concretely, examining the exact sequences gotten from mapping a connected $Y$ into to the fibrations above, we have:
$$\xymatrix{
\cdots \ar[r] & [Y, \GL_1(R_n)] \ar[r] \ar@{=}[d] & [Y, B(S, R_n)]  \ar[r] \ar[d]_-{\theta_r}^\cong & [Y, B\GL_1(S)] \ar[r] \ar@{=}[d] & \cdots \\
\cdots \ar[r] & R_n^0(Y)^\times \ar[r]^-{\psi^r/1} & R_n^0(Y)^\times \ar[r] & [Y, B\GL_1(S)] \ar[r] & \cdots
}$$

\begin{cor}

An $R_n$-oriented Thom spectrum $Y^\zeta$ over connected $Y$ is trivial (i.e., equivalent to $L_{K(n)} \Sigma^\infty Y_+$) if and only if $\theta_r(\zeta) = \psi^r(s)/s$ for some unit $s \in R_n^0(Y)^\times$.

\end{cor}

For $n=1$, this appears to be a strong form of a $K(1)$-local version of Corollary 5.8 of \cite{adams2}.
Further, it says that $Y^\zeta$ is trivial if and only if the cocycle $\theta(\zeta)$ is a coboundary, so the map
$$[Y, B(S, R_n)] \to H^1(\Z_p^{\times}, R_n^0(Y)^{\times}) \; \mbox{ given by } \; \zeta \mapsto \theta(\zeta)$$
is injective.  This is perhaps a form of descent for Picard groups in the $\Z_p^{\times}$-Galois extension $S \to R_n$.

\subsection{A Thom spectrum over $K(\Z_p, n+1)$}

The following proposition, which is immediate from the connectivity of $Y$, is an analogue of the fact that a formal power series is invertible if and only if its constant term is.

\begin{prop}

Let $R$ be an $A_\infty$ ring spectrum.  If $Y$ is connected and $\alpha \in R^0(Y) = [Y, \Omega^\infty R]$, then $\alpha$ lies in $R^0(Y)^{\times} = [Y, \GL_1 R]$ if and only if the restriction of $\alpha$ to a point in $Y$ is a unit; $\alpha|_{\rm pt} \in R_0^{\times}$.

\end{prop}

We will denote by $j$ the natural localisation map $j: X \to R_n = X[\rho^{-1}]$.  We will often confuse $j$ with the corresponding element $j \in R_n^0(K(\Z_p, n+1)) \cong R_n^0(X)$.  Since $j|_{\rm pt} = 1$ is the unit of $R_n$, the previous criterion shows that $j \in R_n^0(X)^{\times}$, while $j-1$ is not a unit.  Consequently, the following is not automatic:

\begin{prop}

Let $r \in \N$ be a topological generator of $\Z_p^{\times}$.  The element $\psi^r(j-1) \in R_n^0(X)$ is divisible by $j-1$, and the quotient $\frac{\psi^r(j-1)}{j-1}$ is a unit.

\end{prop}

\begin{proof}

It follows from the proof of Proposition \ref{action_prop} that $\psi^r(j)$ is simply $j$ composed with the self map of $K(\Z_p, n+1)$ which is multiplication by $r$ in $\pi_{n+1}$.  Furthermore, this is implemented by the $r$-fold diagonal, followed by $r$-ary multiplication in $K(\Z_p, n+1)$.  Since $j$ is multiplicative, we conclude that $\psi^r(j) = j^r$.  Thus  $\psi^r(j-1) = j^r-1$ is evidently divisible by $j-1$, with quotient
$$\frac{\psi^r(j-1)}{j-1} = 1 + j + \cdots + j^{r-1}$$
Restricting $j$ along the unit yields $1$; thus the restriction of $\frac{\psi^r(j-1)}{j-1}$ along the unit gives $r \in R_0^{\times}$, so this class is invertible.

\end{proof}

\begin{defn}

Let $e: K(\Z_p, n+1) \to BX_n = B(S, R_n)_{>0}$ be the unique map satisfying
$$c(\psi^r) \circ e = r^{-1}\frac{\psi^r(j-1)}{j-1} \in R_n^0(X)^{\times}.$$ 

\end{defn}

Here we write $Y_{>m}$ for the $m$-connected cover of $Y$.  Pulling $\gamma$ back over $e$ allows us to define a Thom spectrum $K(\Z_p, n+1)^{\gamma \circ e}$ which, for brevity, we will write as $X^\gamma$.  Then $e$ is defined in such a way that the $r^{\rm th}$ cannibalistic class of $\gamma$ is $\theta_r(\gamma) = r^{-1}\frac{\psi^r(j-1)}{j-1}$.  

In fact, this definition of $e$ is independent of our choice of generator $r$.  An induction using part 4 of Proposition \ref{easy_prop} along with the continuity of the action of the Adams operations gives:

\begin{prop}

For every $k \in \Z_p^{\times}$,
$$\theta_k(\gamma) = k^{-1}\frac{\psi^k(j-1)}{j-1}$$

\end{prop}

Above, the Thom spectrum $X^\gamma$ has been normalised to have Thom class in dimension 0; classically, this corresponds to $(\C P^\infty)^{L-1}$, where $L$ is the tautological line bundle.  We will devote a lot of attention to the Thom spectrum $X^{G\gamma} = G \otimes X^{\gamma}$, the analogue in our setting of $(\C P^\infty)^L = S^2 \wedge (\C P^\infty)^{L-1}$.  We note then that the cannibalistic class $\theta_g(G\gamma)$ satisfies
$$(j-1) \cdot \theta_g(G\gamma) = \psi^g(j-1).$$

%\subsection{A higher J-homomorphism and an analogue of $MU$}
%
%We would like to extend the map $K(\Z_p, n+1) \to B(S, R_n)$ over the inclusion $K(\Z_p, n+1) \to \Omega^\infty R_n$.  Such a map would (as we show below) be a higher-chromatic analogue of the J-homomorphism.  This is possible on the level of spaces, however, the resulting map is not obviously an infinite loop map.  We can, however, construct such an infinite loop map not for $R_n^{\deloc}$:
%
%\begin{defn}
%
%Define the \emph{$n^{\rm th}$ de-localised J-homomorphism} 
%%
%$$J = J_n: R_n^{\deloc} \to b(S, R_n)$$
%%
%to be the composite $J = e \circ \ell: R_n^{\deloc} \to b(S, R_n)$ where $\ell: R_n^{\deloc} \to \Sigma^\infty K(\Z_p, n+1)_+$ is the natural map from the homotopy fibre of $\psi^p$.
%
%Furthermore, let $BX = BX_n := \Omega^\infty_0 R_n^{\deloc}$, and define $MX_n$ as the Thom spectrum 
%%
%$$MX =MX_n := L_{K(n)} BX_n^{\gamma \circ J}$$
%%
%associated to the composite $\xymatrix@1{BX_n \ar[r]^-J & B(S, R_n) \ar[r]^-{\gamma} & B\GL_1 S}$.
%
%\end{defn}
%
%The $K(n)$-localisation map $\Sigma^\infty K(\Z_p, n+1)_+ \to X$ factors through $L_n \Sigma^\infty K(\Z_p, n+1)_+$.  This factorised map lifts to $M_{n, F} X$; its composite with $J$ is the map $e$ of the previous section, by construction. \marginpar{nonsense}

\subsection{A monochromatic J-homomorphism and an analogue of $MU$}

We would like to extend the map $K(\Z_p, n+1) \to B(S, R_n)$ over the inclusion $K(\Z_p, n+1) \to \Omega^\infty R_n$.  This is possible on the level of spaces; however, the result is not obviously an infinite loop map.  If it were, the composite of the associated spectrum maps 
$$(R_n)_{> 0} \to b(S, R_n) \to \Sigma \gl_1 S$$ 
could potentially be the connective cover of a map $(R_n)_{\geq 0} \to \pic (S)$ which provides a higher-chromatic analogue of the J-homomorphism.  Classically (when $n=1$), this is the map which assigns to a vector space $V$ the invertible spectrum $S^V = \Sigma^\infty (V \cup \{\infty\})$.  The possibility that such a map exists is hinted at by Proposition \ref{map_to_pic_prop} above, which in some sense constructs it in $\pi_0$.

In this section, we construct such an infinite loop map; however, we do so at the price of working monochromatically at level $n$.  Let $L_n$ denote localisation with respect to $E(n)$; we recall that the $n^{\rm th}$ monochromatic component of a spectrum $Y$ is fibre $M_n Y$ of the localisation map $L_n Y \to L_{n-1} Y$.  Furthermore, $L_n L_{K(n)} Y = L_{K(n)} Y$.  From this fact and the chromatic fracture square for $Y$, it is easy to see that $M_n Y \simeq M_n L_{K(n)} Y$.  In particular, $M_n \Sigma^\infty K(\Z_p, n+1)_+ \simeq M_n X$.

Let $A$ be a $K(n)$-local $E_\infty$ ring spectrum.  We recall from Theorem 4.11 of \cite{ahr} that the homotopy groups of the homotopy fibre $F$ of the localisation map $\gl_1 A \to L_n \gl_1 A$ (the ``discrepancy spectrum") are torsion and concentrated in dimensions $q \leq n$.  The $p$-completion of this fibre sequence is then of the form
$$F^{\wedge}_p \to (\gl_1 A)^{\wedge}_p \to (L_n \gl_1 A)^{\wedge}_p = L_{K(1) \vee \dots \vee K(n)} \gl_1 A$$
and the homotopy of $F^{\wedge}_p$ is now concentrated in dimensions $q \leq n+1$.  Furthermore, the map $\gl_1 A \to (\gl_1 A)^{\wedge}_p$ is an equivalence on connected covers, since the positive homotopy of $\gl_1 A$ is isomorphic to that of $A$ (which is $p$-complete).

Take $A = R_n$, and consider the diagram
$$\xymatrix{
\Sigma^\infty K(\Z_p, n+1)^{\wedge}_p  \ar@{.>}[d]_-{\alpha} \ar[rrr]^-{k^{\wedge}_p} \ar[drr]_-{\loc} \ar@{.>}[dr]_-{\overline{\loc}} & & & (\gl_1 R_n)^{\wedge}_p \ar[d] \\
Q \ar@{.>}[d]_-{\beta} & (M_n \Sigma^\infty K(\Z_p, n+1))^{\wedge}_p \ar@/^/[d]^-{M_n (a)} \ar@/_/[dl]_-{M_n(ja)} \ar[r]^-{\simeq} & (L_n \Sigma^\infty K(\Z_p, n+1))^{\wedge}_p \ar[r]^-{L_n k} & (L_n \gl_1 R_n)^{\wedge}_p \ar[d] \\
(M_n R_n)^{\wedge}_p \ar@/^/[r]^-{M_n(\ell)} & (M_n X)^{\wedge}_p \ar@/^/[u]^-{M_n (b)} \ar@/^/[l]^-{M_n(j)} & & \Sigma F^{\wedge}_p.
}$$
%

%Let $F_{<n}$ be the Postnikov truncation of $F$ with homotopy in dimensions $q<n$, so that there is a fibration sequence
%%
%$$\Sigma^n H\pi_n(F) \to F \to F_{<n}$$
%
%For a spectrum $Y$, we will write $Y_{>n}$ for the $n$-connected cover of $Y$, with homotopy groups beginning in dimension $n+1$.  Now take $A = R_n$, and consider the following diagram \marginpar{receives a map from $K(\Z, n+1)$}
%%
%$$\xymatrix{
% & M_{n, F} X \ar[d] \ar@{.>}[rrr]^-{K} & & & \gl_1 R_n \ar[d] & \\
%(M_n X)_{>n+1} \ar[r] \ar[ur] & (M_n X)_{>n} \ar@{.>}[drr]_-f \ar[r] & M_n X \ar[r]^-{k} & M_n \gl_1 R_n \ar[r] & L_n \gl_1 R_n \ar[d] & \\
% & & F_{<n} \ar[r] & \Sigma^{n+1} H\pi_n(F) \ar[r] & \Sigma F \ar[r] & \Sigma F_{<n}.
%}$$
%%
%%
%$$\xymatrix{
%M_{n} R_n \ar[d]_-{M_n(j^{-1})} & M_{n, F} R_n \ar[l] \ar@{.>}[rr]^-{K}  & & \gl_1 R_n \ar[d] & \\
%(M_n X)_{>n} \ar@{.>}[drr]_-f \ar[r] & M_n X \ar[r]^-{k} & M_n \gl_1 R_n \ar[r] & L_n \gl_1 R_n \ar[d] & \\
%& F_{<n} \ar[r] & \Sigma^{n+1} H\pi_n(F) \ar[r] & \Sigma F \ar[r] & \Sigma F_{<n}.
%}$$
%%
Here $k: \Sigma^\infty K(\Z_p, n+1) \to \gl_1 R_n$ is adjoint to the map 
$$r^{-1}\frac{\psi^r(j-1)}{j-1}: K(\Z_p, n+1) \to \GL_1 R_n$$
of the previous section.  It follows from $\cite{rw}$ that $L_{n-1} \Sigma^\infty K(\Z_p, n+1) = L_{\Q} \Sigma^\infty K(\Z_p, n+1)$, so their common $p$-completion is null.  Therefore, the $p$-complete localisation map 
$$\loc: \Sigma^\infty K(\Z_p, n+1)^{\wedge}_p \to (L_n \Sigma^\infty K(\Z_p, n+1))^{\wedge}_p$$
lifts to $(M_n \Sigma^\infty K(\Z_p, n+1))^{\wedge}_p$.  The map $a: \Sigma^\infty K(\Z_p, n+1) \to \Sigma^\infty K(\Z_p, n+1)_+$ is the inclusion of the wedge summand, and $b$ is its one-sided inverse.  Similarly, the one-sided inverse $\ell: R_n \to X$ to $j$ induces a splitting of $M_n X$.

\begin{defn} Let $Q$ be the $n+2^{\rm nd}$ Moore-Postnikov factorisation of the map $M_n(ja) \circ \overline{\loc}$.  \end{defn}

Then $Q$ is $n$-connected, $\alpha_*: \Z_p = \pi_{n+1}\Sigma^\infty K(\Z_p, n+1)^{\wedge}_p \to \pi_{n+1} Q$ is an isomorphism, and $\alpha$ and $\beta$ induce an isomorphism
$$\pi_{n+2} Q \cong \im[\pi_{n+2}\Sigma^\infty K(\Z_p, n+1)^{\wedge}_p \to \pi_{n+2} (M_n R_n)^{\wedge}_p].$$
Further, $\beta: Q \to M_n R_n$ is an isomorphism in homotopy in degrees greater than $n+2$.  As such, we regard $Q$ as a stand-in for $M_n R_n$ (and hence $R_n$) which interpolates between it and the spectrum $\Sigma^\infty K(\Z_p, n+1)^{\wedge}_p$.  In particular, since $K(n)$-localisation turns co-connected maps into equivalences, we have that
$$\beta: L_{K(n)} Q \to L_{K(n)} M_n R_n \simeq R_n$$
is an equivalence.

%The map from $(M_n X)_{>n}$ to $\Sigma F_{<n}$ along solid arrows is null, since the target has homotopy groups concentrated in dimensions $q\leq n$.  Thus it lifts to the map $f$; the lift is unique since $F_{<n}$ is $n$-coconnnected. 

%The splitting $(\psi^p \vee j): X \simeq X \vee R_n$ yields the map $M_n(j^{-1}): M_n R_n \to M_n X$.  Define $M_{n, F} R_n$ as the homotopy fibre of the composite
%%
%$$\xymatrix@1{M_n R_n \ar[rr]^-{M_n(j^{-1})} & & (M_n X)_{>n} \ar[r]^-{f} & \Sigma^{n+1} HB}$$
%%
%where $B \leq \pi_{n} F$ is the image of $\pi_{n+1} M_n R_n$ under $M_n(j^{-1}) \circ f$.

Write $k'$ for the map $k' = L_n k \circ M_n(b \ell) \circ \beta$; then we have the commuting diagram of cofibre sequences
$$\xymatrix{
\Sigma^\infty K(\Z_p, n+1)^{\wedge}_p \ar[r]^-{k^{\wedge}_p} \ar[d]_-{\alpha} & (\gl_1 R_n)^{\wedge}_p \ar[d] \\
Q \ar[r]^-{k'} \ar@{.>}[ur]^-{K} \ar[d] & (L_n \gl_1 R_n)^{\wedge}_p \ar[d] \\
\cofib(\alpha) \ar[r] & \Sigma F^{\wedge}_p
}$$

\begin{prop}

There exists a map $K: Q \to (\gl_1 R_n)^{\wedge}_p$ making the above diagram commute.  Further, the set of such maps is a torsor for $\pi_{n+1}(F^{\wedge}_p)$.

\end{prop}

\begin{proof}

The definition of $Q$ is such that $\cofib(\alpha)$ is $(n+2)$-connected.  In contrast, $\Sigma F^{\wedge}_p$ is $(n+2)$-co-connnected, so the bottom map is in the diagram is necessarily null, providing the existence of a map $K$.  Further, the set of such lifts is a torsor for the group
$$[Q, F^{\wedge}_p] = \Hom[\pi_{n+1} Q, \pi_{n+1}(F^{\wedge}_p)] = \Hom[\Z_p, \pi_{n+1}(F^{\wedge}_p)] =  \pi_{n+1}(F^{\wedge}_p),$$
where we again employ the co-connectivity of $F^{\wedge}_p$.

\end{proof}

Note that since $Q$ is connected and $\gl_1 R_n \to (\gl_1 R_n)^{\wedge}_p$ is an equivalence on connected covers, we may regard any choice of $K$ above as actually mapping into $\gl_1 R_n$.

\begin{defn}

For any choice of $K$ above, we define the \emph{$n^{\rm th}$ monochromatic J-homomorphism} 
$$J = J_n: Q \to b(S, R_n)$$
to be the unique map with $c(\psi^r) \circ J = K$.  

Furthermore, let $BX = BX_n := \Omega^\infty Q$, and define $MX_n$ as the Thom spectrum 
$$MX =MX_n := L_{K(n)} BX_n^{\gamma \circ J}$$
associated to the composite $\xymatrix@1{BX_n \ar[r]^-J & B(S, R_n) \ar[r]^-{\gamma} & B\GL_1 S}$.

\end{defn}

It is unclear to us the extent to which the choice of the lift $K$ affects these constructions.  We will make little use of the spectrum $MX_n$ in this paper; when we do, all our results will be insensitive to the choice of $K$.  We notice that $\gamma \circ J$ is an infinite loop map.  Using the methods of \cite{lewis_may_steinberger}, we conclude:

%We do have the following robustness result, though:
%
%\begin{prop} \label{J_robust_prop}
%
%For any choice of $K$, the $K(n)$-localisation of the associated J-homomorphism $L_{K(n)} J: L_{K(n)} Q \to L_{K(n)} b(S, R_n)$ is an equivalence.
%
%\end{prop}
%
%\begin{proof}
%
%In the diagram below, $\log$ denotes Rezk's logarithm, introduced in \cite{rezk}; it is a $K(n)$-local equivalence.
%%
%$$\xymatrix{
%L_{K(n)} Q \ar[r]^-{K} \ar[d]_-{\beta} & L_{K(n)} \gl_1 R_n \ar[r]^-{c(\psi^r)} \ar[d]_-{\log} & L_{K(n)} b(S, R_n) \\
%R_n \simeq L_{K(n)} M_n R_n \ar[r]_-{\log \circ K} & R_n & 
%}$$
%The cannibalistic class map $c(\psi^r)$ is an equivalence of connective covers, and so is an equivalence on $K(n)$-localisations.  Similarly, the localisation of $\beta$ is an equivalence.  Thus the result follows if the lower left map $\log \circ K$ is an equivalence.  However, since $K$ extends $k^{\wedge}_p$, the following diagram is commutative:
%%
%$$\xymatrix{
%L_{K(n)} \Sigma^\infty K(\Z_p, n+1) \ar[r]^-{\log \circ k} \ar[d]_-{L_{K(n)} \alpha} & R_n \\
%L_{K(n)} Q \ar[ur]_-{\log \circ K}. &  
%}$$
%%
%Further, $\log \circ k$ is gotten from the composite map of based spaces
%%
%$$\xymatrix@1{Q K(\Z_p, n+1) \ar[rr]^-{Q(r^{-1}\frac{\psi^r(j-1)}{j-1})} & & \GL_1 R_n \ar[r]^-{s} & \Omega^\infty R_n}$$
%%
%(where $s$ is the basepoint shift $x \mapsto x-1$) by applying the Bousfield-Kuhn functor $\phi_n$.
%
%
%\end{proof}

\begin{prop}

The Thom spectrum $MX_n$ is an $E_\infty$-ring spectrum.

\end{prop}

Note also that by its very nature, $\gamma \circ J$ is $R_n$-oriented, yielding a Thom isomorphism 
$$R_n^{\bigstar}(MX) \cong R_n^\bigstar BX$$
of $R_n^\bigstar BX$-modules.  

\begin{exmp} The notation $MX_n$ is meant to evoke thoughts of the complex cobordism Thom spectrum $MU = BU^\gamma$ (where $\gamma$ is the tautological virtual bundle over $BU$).  Indeed, this is precisely the case when $n=1$ (at least $K(1)$-locally).  There is an equivalence 
$$R_1= X_1[\rho^{-1}] = L_{K(1)} \Sigma^\infty \C P^\infty_+ [ \beta^{-1}] \simeq L_{K(1)} K \simeq K^{\wedge}_p,$$
by Snaith's theorem.  Then $M_1 R_1$ is the first monochromatic component (or $\Q$-acyclisation) of $K^{\wedge}_p$; this spectrum has homotopy only in odd degrees, where it is $\Q_p / \Z_p$.  The monochromatisation is undone by the $p$-completion: $(M_1 R_1)^{\wedge}_p = K^{\wedge}_p$.  Finally, examining the homotopy of $Q$, we see that the resulting spectrum $Q \simeq (K^{\wedge}_p)_{>0}$ is equivalent to the $p$-completion of the connected cover of $K$.  Thus its infinite loop space is $BX_1 = BU^{\wedge}_p$; and indeed the natural map of $K(1)$-local Thom spectra $L_{K(1)} MU \to MX_1$ is an equivalence by the Thom isomorphism. 

%
%sits in a fibration sequence 
%%
%$$K(\pi_1(F), 1) \to BX_1 \to \Omega^\infty (M_1 R_1)_{>1}.$$
%%
%In this case, $F$ is the homotopy fibre of the map $\gl_1 K^{\wedge}_p \to L_1 \gl_1 K^{\wedge}_p$.
%
%One can show that at odd primes, $\pi_1(F) = \Q_p/\Z_p$.  To see this, examine the long exact sequence
%%
%$$\dots \to \pi_2 \gl_1 K^{\wedge}_p \to \pi_2 L_1 \gl_1 K^{\wedge}_p \to \pi_1 F \to 0$$
%
%Thus the group $B$ is 0, and so $BX_1 \simeq \Omega^\infty (M_1 R_1)_{>1}$.  Further, the map $BX_1 \to BU^{\wedge}_p$ is an isomorphism in $K(1)_*$, as one can see from the $K(1)_*$-based Rothenberg-Steenrod spectral sequence for the fibre sequence of infinite loop spaces:
%%
%\beqn \Omega^{\infty+1} K_{\Q_p}^{\wedge} \to BX_1 \to BU^{\wedge}_p \label{Q/Z_eqn} \eeqn

In this case, there is a choice of $K$ above so that the monochromatic J-homomorphism $J_1: BX_1 \to B(S, R_1)$ is precisely the $p$-completion of the usual J-homomorphism.

%Equation (\ref{Q/Z_eqn}) computes $\pi_* BX_1$: it is $0$ when $*$ is even, and $\Q_p / \Z_p$ when $*$ is odd and greater than 1.  The composite $\gamma \circ J_1: BX_1 \to B\GL_1(S)$ is null in $\pi_*$, since the homotopy of the codomain is concentrated in even degrees. \marginpar{more here}

\end{exmp}

In summary, we have maps
$$\xymatrix@1{K(\Z_p, n+1) \ar[r]^-{\alpha} & \Omega^\infty Q \ar[r]^-{J} \ar[r] & B(S, R_n) \ar[r]^-{\gamma} & B\GL_1(S)}$$
where all but the first are infinite loop maps.  The construction of $K$ makes it evident that for any choice, the resulting map $J$ yields the map $e$ of the previous section upon composite with $\alpha$.  This mirrors the classical situation for $n=1$:
$$\xymatrix@1{\C P^\infty = BU(1) \times \{1\} \ar[r]^-{\subseteq} & BU \times \Z \ar[r]^-{J} \ar[r] & B(S, K) \ar[r]^-{\gamma} & B\GL_1(S^0)}$$
The full composite defines the Thom spectrum of the tautological line over $\C P^\infty$; starting at the second term yields $MU$.

\subsection{A zero section map}

Recall that the zero section of the Thom spectrum $\Sigma^2 MU(1) = \Sigma^2 (\C P^\infty)^{L-1}$ is an equivalence.  The goal of this section is to give an analogous result the current setting: a stable ``zero-section" $z:X \to X^{G\gamma}$ which is a $K(n)$-local equivalence.

\begin{defn} 

For an $R_n$-oriented Thom spectrum $Y^{ \zeta}$ over $Y$, define $f_g( \zeta)$ as the composite
$$\xymatrix{
R_n \otimes Y  \ar[r]^-{1 \otimes \Delta} &  R_n \otimes Y \otimes Y  \ar[rr]^-{\psi^g \otimes \theta_g( \zeta)^{-1} \otimes 1} & & R_n \otimes R_n \otimes Y \ar[r]^-{\mu \otimes 1} & R_n \otimes Y
}$$

\end{defn}

\begin{prop} \label{f_g_prop}

Let $Y^{G\zeta}$ be $R_n$-oriented with Thom class $u$ (and associated Thom isomorphism $T_u$). Then the following diagram commutes:
$$\xymatrix{
R_n \otimes Y^{G \zeta} \ar[r]^-{T_{u}}  \ar[d]_-{\psi^g \otimes 1} & G \otimes R_n \otimes Y  \ar[d]^-{1\otimes f_g(\zeta)}  \ar[r]^-{\delta \otimes 1} & R_n \otimes Y \ar[d]^-{g^{-1} f_g(\zeta)}   \\
R_n \otimes Y^{G \zeta} \ar[r]^-{T_{u}} &  G \otimes R_n \otimes Y \ar[r]^-{\delta \otimes 1} & R_n \otimes Y,
}$$

\end{prop}

\begin{proof}

It follows quickly from this definition and the fact that $\psi^g \circ \delta  = g \cdot \delta$ that the right square commutes.  Thus the result follows from the commutativity of the left square without the tensor factor of $G$.  Here is an expanded version of that square, which commutes by inspection:
$$\xymatrix{
R_n \otimes Y^{ \zeta} \ar[r]^-{1 \otimes D} \ar[dr]_-{1 \otimes D}  \ar[ddd]_-{\psi^g \otimes 1} & R_n \otimes Y^{ \zeta} \otimes Y \ar[rr]^-{1\otimes u \otimes 1} \ar[dr]_-{1 \otimes 1 \otimes \Delta} & & R_n \otimes R_n \otimes  Y \ar[r]^-{\mu \otimes 1} \ar[d]^-{1 \otimes 1 \otimes \Delta}  & R_n  \otimes Y  \ar[d]^-{1 \otimes \Delta}    \\
 & R_n \otimes Y^{ \zeta} \otimes Y  \ar[dd]_-{\psi^g \otimes 1 \otimes 1} \ar@/_1pc/[ddrr]_-{\psi^g \otimes  u \otimes 1} \ar[r]_-{1 \otimes D \otimes 1} & R_n \otimes Y^{ \zeta} \otimes Y \otimes Y  \ar[r]^-{1\otimes u \otimes1 \otimes 1} \ar@/^0pc/[dr]_-{\psi^g \otimes \theta_g(\zeta) \cdot u \otimes \theta_g(\zeta)^{-1} \otimes 1 \hspace{0.7cm}} & R_n \otimes R_n \otimes  Y \otimes Y \ar[r]^-{\mu \otimes 1 \otimes 1}  \ar[d]^-{\psi^g \otimes \psi^g \otimes \theta_g(\zeta)^{-1} \otimes 1} & R_n \otimes Y \otimes  Y \ar[d]^-{\psi^g \otimes \theta_g(\zeta)^{-1} \otimes 1} \\
&  & &  R_n \otimes R_n \otimes R_n \otimes Y \ar[r]^-{\mu \otimes 1 \otimes 1} \ar[d]^-{1 \otimes \mu \otimes 1} &  R_n \otimes R_n \otimes  Y  \ar[d]^-{\mu \otimes 1} \\
R_n \otimes Y^{ \zeta} \ar[r]^-{1 \otimes D} & R_n \otimes Y^{ \zeta} \otimes Y \ar[rr]^-{1\otimes u \otimes 1} & & R_n \otimes R_n \otimes   Y \ar[r]^-{\mu \otimes 1} & R_n \otimes Y.
}$$

\end{proof}

Since the homotopy fibre of $\psi^g - 1: R_n \to R_n$ is $S$, we see that $\hofib((\psi^g-1) \otimes 1) =  Y^{G\zeta}$ in the diagram in the statement of Proposition \ref{f_g_prop}.  Since the horizontal maps are equivalences, we conclude:

\begin{cor}

$\hofib(g^{-1} f_g(\zeta)-1) = \hofib(f_g(\zeta) - g) = Y^{G\zeta}$.

\end{cor}

We now focus on $X^{G\gamma}$.  Define a map $\overline{z}: X \to R_n \otimes X$ by $\overline{z} = ((j-1) \otimes 1) \circ \Delta$.  Since 
$$\psi^g(j-1) = \theta_g(G\gamma) \cdot (j-1) = g \theta_g(\gamma)\cdot (j-1),$$ 
we see that this commutes
$$\xymatrix{
X \ar[rr]^-{\Delta} \ar[dr]_-{\Delta} \ar[dd]_-{g\overline{z}} & & X \otimes X \ar[r]^-{(j-1) \otimes 1} \ar[d]^-{1 \otimes \Delta} & R_n \otimes X \ar[d]^-{1 \otimes \Delta} \\
 & X \otimes X \ar[dl]^-{g(j-1) \otimes 1} \ar[r]^-{\Delta \otimes 1}& X \otimes X \otimes X \ar[r]^-{(j-1) \otimes 1} \ar[dr]_-{\hspace{-2cm}g \theta_g(\gamma) (j-1) \otimes \theta_g(\gamma)^{-1} \otimes 1} & R_n \otimes X \otimes X \ar[d]^-{\psi^g \otimes \theta_g(\gamma)^{-1} \otimes 1} \\
R_n \otimes X & & & R_n \otimes R_n \otimes X \ar[lll]^-{\mu \otimes 1}
}$$
Passage along the right side of this diagram gives $f_g(\gamma) \circ \overline{z}$, so $f_g(\gamma) \circ \overline{z} = g\overline{z}$.  Thus $\overline{z}$ lifts to a map\footnote{The map $z$ is, of course, not uniquely specified by this computation.  However, any lift of $\overline{z}$ will serve for our purposes, as will be evident from the proof of Proposition \ref{Euler_prop}.}
$$z: X \to X^{G\gamma}.$$
which we will regard as the zero-section of the Thom spectrum.

The following result encourages us to regard $j-1$ as the Euler class of $\gamma$ over $X$.  Indeed, when $n=1$, it is precisely the K-theoretic Euler class of the tautological line bundle over $\C P^\infty$.  We continue to write $e^*(u) \in R_n^G(X^{G\gamma})$ for the Thom class of $X^{G\gamma}$, and $e^*(\ubar) \in R_n^0(X^{G\gamma})$ for its normalisation.

\begin{prop} \label{Euler_prop}

$z^*(e^*(\ubar)) = j-1$.

\end{prop}

\begin{proof}

We first observe that one may derive $e^*(\ubar)$ from the fibre sequence for $f_g(\gamma)-g$; that is, this commutes:
$$\xymatrix{
 & X^{G\gamma} \ar[d] \ar[dr]^-{e^*(\ubar)} & \\
X \ar[ur]^-z \ar[r]^-{\overline{z}} & R_n \otimes X \ar[r]^-{1 \otimes p} \ar[d]_-{f_g(\gamma)-g} & R_n \otimes S = R_n \\
 & R_n \otimes X & 
}$$
Here $p: X \to S$ is induced by the projection of $K(\Z_p, n+1)$ to a point.  The commutativity of the upper right triangle follows from the diagram of fiber inclusions
$$\xymatrix{
S \otimes X^{G\gamma} = X^{G\gamma} \ar[rrrr]^-= \ar[d]_-{\eta \otimes 1} & & & & X^{G\gamma} \ar[d] & \\
R_n \otimes X^{G \gamma} \ar[r]^-{1 \otimes D}  \ar[drrr]_-{1 \otimes e^*(u)} & R_n \otimes X^{G \gamma} \otimes X \ar[r]^-{1\otimes e^*(u) \otimes 1} & R_n \otimes R_n \otimes G \otimes X \ar[r]^-{\mu \otimes 1} & R_n \otimes G \otimes X  \ar[r]^-{\delta \otimes 1} \ar[d]^-{1 \otimes 1\otimes p} & R_n  \otimes X \ar[d]^-{1 \otimes p} \\
 & & & R_n \otimes G \ar[r]^-\delta & R_n, 
}$$
since passage along the lower left defines $\delta e^*(u) = e^*(\ubar)$.  Then
$$z^*(e^*(\ubar)) = (1 \otimes p) \circ \overline{z} = (1 \otimes p) \circ ((j-1) \otimes 1) \circ \Delta = j-1$$

\end{proof}

Recall that $K^*(X) = K_*[[y]]$, and the natural transformation $t: R_n \to K$ of Proposition \ref{t_prop}.  We note that  $t_*(e^*(u))$ is the $K$-Thom class of $X^{G\gamma}$.  It satisfies $z^*(t_*(e^*(u))) = y$, for
$$z^*(t_*(e^*(u))) = t \circ (z^*(e^*(u))) = t \circ(j-1) = (1+y) - 1 = y.$$
We conclude:

\begin{prop} \label{euler_prop}

The composite of the Thom isomorphism for $X^{G\gamma}$ and $z^*$, as a map from $K^*(X) \cong K^{*+2g(n)}(X^{G\gamma})$ to $K^{*+2g(n)}X$ is multiplication by $y$.

\end{prop}

Note that the image of $z^*$ may be identified as the (split) subspace $y K_*[[y]] = \tilde{K}^*(K(\Z_p, n+1))$.  Since $z^*$ is evidently injective, we obtain the main result of this section:

\begin{cor}

The zero section restricts to a $K(n)$-local equivalence $z: L_{K(n)} \Sigma^\infty K(\Z_p, n+1) \to X^{G\gamma}$.

\end{cor}

\section{Higher orientations for chromatic homotopy theory}

\subsection{$n$-orientations and formal group laws}

The map $\rho: G \to X$ allows us to extend the notion of a complex orientation of a cohomology theory to the $K(n)$-local category:

\begin{defn}

An \emph{$n$-orientation} of a $K(n)$-local ring spectrum $E$ is a class 
$$x \in E^{G}(X) = [X, G \otimes E]$$
with the property that $\rho^*(x) \in E^{G}(G) = \pi_0 E = E_{0}$ is a unit.  

\end{defn}

We note that since $\rho_1: L_{K(1)} S^2 \to X_1$ is the $K(1)$-localisation of the inclusion $\C P^1 \subseteq \C P^\infty$, a $1$-orientation is precisely a $K(1)$-local complex orientation.

\begin{exmps} The following are $n$-oriented spectra:

\begin{enumerate}

\item The spectrum $R_n = X[\rho^{-1}]$ is naturally $n$-oriented.  Define $x \in R_n^{G}(X)$ by
$$x = \delta^{-1} \otimes (j-1): X \to G \otimes X[\rho^{-1}].$$
Then $\rho^*(x) = \delta^{-1}\cdot \rho^*(j-1) = \delta^{-1}\cdot \delta = 1$.

\item $K$ is $n$-oriented, via the conveniently named class $x \in K^{2g(n)}(X) = K^G (X)$.  The image of $\rho$ in $K_*(X_n)$ is $b_0$, and so $\rho^*(x) = x(b_0) = 1$.  More generally, any power series $f(x) \in K_*[[x]]$ which begins with a unit multiple of $x$ gives an $n$-orientation of $K$.

\item $E_n$ is $n$-oriented by a lift of the previous orientation.  More carefully, 
$$E_n^G(X) \cong E_n^{2g(n)}(X) = u^{g(n)} \cdot E_n^0(X) = u^{g(n)} \cdot \W(\F_{p^n})[[u_1, \dots, u_{n-1}]][[\Z_p]]$$
Then an orientation is given by the class of $u^{g(n)} \cdot g$, since the image of a fundamental class under $\rho$ in ${E_n}_*(X) \cong C(\Z_p, {E_n}_*)$ is a function which carries $g$ to a unit.

\item $MX_n$ is $n$-oriented via the map $x:=(\alpha^\gamma \otimes 1) \circ z$, as can bee seen from the diagram:
$$\xymatrix{
X_n \ar[r]^-z & X^{G\gamma} \ar[r]^-=  & X^{\gamma} \otimes G \ar[r]^-{\alpha^\gamma \otimes 1} & MX_n \otimes G \\
G \ar[u]_-\rho & & S \otimes G \ar[ll]^-= \ar[u]^-{\eta \otimes 1}
}$$
which commutes by the proof of Proposition \ref{Euler_prop}. Thus $\rho^*(x) = \alpha^\gamma \circ \eta$ is the unit of the ring spectrum $MX_n$.

\end{enumerate}

\end{exmps}

\begin{thm} \label{orient_thm}

An $n$-orientation $x$ of $E$ gives an isomorphism $E^{\bigstar}(X) \cong E_\bigstar[[x]]$, and the multiplication in $X$ equips this ring with a formal group law $F(x, y) \in E_\bigstar[[x, y]]$.

\end{thm}

\begin{proof}

We note that an $n$-orientation of $E$, $x \in E^G(X)$ defines, via the equivalence 
$$z: L_{K(n)} \Sigma^\infty K(\Z_p, n+1) \to X^{G\gamma},$$
a Thom class $u = (z^*)^{-1}(x) \in E^G( X^{G\gamma})$, since restricting $u$ to each fibre gives the same class as restricting $x$ to $G$, namely, a unit.  Furthermore, extending $z$ to all of $X =  L_{K(n)} \Sigma^\infty K(\Z_p, n+1) \vee S$, we see that the augmentation ideal of $E^\bigstar X$ is $\tilde{E}^\bigstar K(\Z_p, n+1) = z^* E^\bigstar X^{G\gamma}$.  Moreover, by the Thom isomorphism, the latter is the cyclic ideal generated by $z^*u = x$.

We conclude two things: that the augmentation ideal of $E^\bigstar X$ is generated by $x$, and that it is itself isomorphic to $E^\bigstar X^{G\gamma}$, and hence $E^\bigstar X$ via the Thom isomorphism.  We thereby inductively observe that the quotients of the filtration by powers of the augmentation ideal are free $E^\bigstar$-modules of rank 1 generated by the powers of $x$, and so $E^\bigstar X \cong E_\bigstar[[x]]$.

The K\"unneth spectral sequence for $E^\bigstar(X \otimes X)$ collapses to $E^\bigstar(X_n) \otimes_{E_\bigstar} E^\bigstar(X_n) = E_\bigstar[[x, y]]$, since each factor has free $E^\bigstar$-cohomology. The properties of the formal group law all derive from the unital, associative multiplication on $X$.

\end{proof}

\subsection{A remark on $(n-1)$-gerbes}

It is natural to ask what sort of object an $n$-orientation is orienting.  We recall that a complex orientation of $E$ yields a theory of Chern classes for $E$ in which the orientation is the first Chern class of the tautological bundle.  Furthermore, the formal group law of the cohomology theory encodes the Chern class of a tensor product of line bundles for the cohomology theory.  

We will consider a \emph{$p$-adic $(n-1)$-gerbe} $V$ over a space $Y$ to be a map $f_V: Y \to K(\Z_p, n+1)$.  Here we are purposefully confusing a gerbe with its Dixmier-Douady type of characteristic class.  For an $n$-oriented cohomology theory $E$, the orientation class $x$ defines a first Chern class for $V$ by the formula
$$c_1(V) := f_V^*(x) \in E^{\Sdet}(Y).$$
The formal group law on $E^\bigstar$ then allows one to compute $c_1(V \otimes W) = F(c_1(V), c_1(W))$.

It is not apparent to the author how to extend this notion to higher rank (i.e., non-abelian) $(n-1)$-gerbes or higher Chern classes.

\subsection{Multiplicative $n$-oriented spectra}
\begin{defn}

An $n$-oriented spectrum $R$ is \emph{multiplicative} if there is a unit $t \in R_\bigstar$ such that the formal group law that $R$ supports on $R^\bigstar(X_n)$ is given by the formula
$$F(x, y) = x + y + txy$$

\end{defn}

\begin{thm} \label{s_o_thm}

The spectrum $R_n$ with its natural $n$-orientation is the universal (i.e., initial) multiplicative $n$-oriented spectrum.

\end{thm}

\begin{proof}

This is argument closely follows that of Spitzweck-{\O}stv{\ae}r \cite{spitz_ost} for the motivic analogue of Snaith's theorem.

First, $X_n[\rho^{-1}]$ is multiplicative.  We note that $j$ is a map of ring spectra, as it is induced by the identity on $X_n$.  Thus if $m$ denotes multiplication in $X_n$, 
$$m^*(j) = j \otimes j \in R_n^\bigstar(X_n \otimes X_n) = {R_n}_\bigstar[[x, y]],$$
since the tensor product is multiplication in the cohomology of the smash product.  Definitionally, $j\otimes 1 = 1+\delta x$, and $1 \otimes j = 1+\delta y$, so $j \otimes j = 1+\delta(x+y+\delta xy)$.  Thus
$$F(x, y) = m^*(x) = m^*(\delta^{-1}(j-1)) = \delta^{-1}(j \otimes j-1) = x+y+\delta xy$$

Loosely, $R_n = X_n[\rho^{-1}]$ is universal because $j$ is initial amongst maps from $X_n$ to spectra in which $\rho$ is invertible.  More carefully, let $E$ be an $n$-oriented, multiplicative spectrum, with orientation $v$, and whose formal group satisfies $F(v, w) = v + w + tvw$ for some $t \in E_G$.  Then there is a map of ring spectra\footnote{We are not claiming that this map is highly structured, only that it preserves multiplication and units up to homotopy.} $\phi: X_n \to E$ defined by $\phi = 1+tv$.  To check that $\phi$ is multiplicative, we need to see that $m^*(\phi) = \phi \otimes \phi \in E^\bigstar(X_n \otimes X_n) = E_\bigstar[[x, y]]$.  But
$$m^*(\phi) = 1+tm^*(v) = 1+t F(v, w) = 1+ tv + tw +t^2 vw = (1+tv)(1+tw) = \phi \otimes \phi.$$
Similarly, $\phi$ is unital, since $t$ restricts to $0$ over $S^0$. 

We note that 
$$\phi_*(\rho) = (1+tv) \circ \rho = tv \circ \rho = t \cdot (\rho^*v)$$
is a product of units, so $\phi$ extends over $j$ to a map of ring spectra $\Phi: R_n = X_n[\rho^{-1}] \to E$.  

Since $E$ is $n$-oriented, the map $\delta: G \to X_n$ defines a function 
$$\delta^*: E_\bigstar[[v]]_G  = E^G(X_n) \to E^G(G) = E_0$$
which carries a power series in $v$ to the coefficient of $v$.  Therefore $\phi \circ \delta = \delta^*(\phi) = t$.  We see, then, that $\Phi$ is orientation preserving (i.e., $\Phi_*(x) = v$), since
$$v = t^{-1}(1+tv - 1) = \Phi_*(\delta^{-1})(\phi - 1) = \Phi_*(\delta^{-1}(j - 1)) = \Phi_*(x).$$

Let $\Psi: X_n[\rho^{-1}] \to E$ is any other orientation-preserving map.  Since it is orientation preserving, it must preserve the formal group law, so $\Psi_*(\delta) = t$.  Consider the composite map $\psi := \Psi \circ j: X_n \to E$.  Then 
$$v = \Psi_*(x) = \Psi_*(\delta^{-1}(j-1)) = t^{-1}(\psi-1),$$
giving $\psi = 1+tv = \phi$, and so $\Psi = \Phi$.

Therefore, for any multiplicative, $n$-oriented spectrum $E$, there exists a orientation-preserving map $\Psi: X_n[\rho^{-1}] \to E$, unique up to homotopy.

\end{proof}

\subsection{Identifying the coefficients of $R_n$}

There is an ``integral lift" $\W(K)$ of the cohomology theory $K$ with homotopy groups 
$$\W(K)_* = \W(\F_{p^n})[u^{\pm 1}]$$
One may define $\W(K)$ as the $E_n$-algebra
$$\W(K) := E_n/(u_1, \dots, u_{n-1}).$$
Note that reduction modulo $p$ gives a natural transformation $\W(K) \to K$. Furthermore, the reduction map $E_n \to \W(K)$, being a ring homomorphism, carries an $n$-orientation of $E_n$ to an $n$-orientation for $\W(K)$.

\begin{prop}

The formal group law on $\W(K)^{G^*}(X)$ is the universal multiplicative formal group law in the category of $\W(\F_{p^n})$-algebras.

\end{prop}

\begin{proof}

$K$ supports the multiplicative formal group over $\F_{p^n}$, via Proposition \ref{rw_prop}; $\W(K)$ is its universal deformation.  But the universal deformation of the multiplicative group is the multiplicative group over $\W(\F_{p^n})$, which is clearly initial in the indicated category.

\end{proof}

Recall that the Picard-graded homotopy of $E_n$ associated to powers of $G$ are 
$$(E_n)_{G^*} = \W(\F_{p^n})[[u_1, \dots, u_{n-1}]][s^{\pm 1}],$$
where $|s| = \dim(G) = 2g(n)$, following the discussion in Example \ref{k_invert_exmp}.  Thus $\W(K)_{G^*} = \W(\F_{p^n})[s^{\pm 1}]$.

We note that $\Z_p^{\times}$ acts on the multiplicative group $\G_m$, and therefore on $\W(K)_{G^*}$, since the previous Proposition implies that $\W(K)_{G^*}$ co-represents $\G_m$ in the category of $\W(\F_{p^n})$-algebras.  The action (through $\W(\F_{p^n})$-algebra homomorphisms) is easily seen to be determined by the formula $\gamma \cdot s = \gamma s$.

For the next result, we recall that $R_n$ is also equipped with an action of $\Z_p^{\times} = \G_n / \SGn$:

\begin{cor} \label{rep_cor}

As a $\Z_p^{\times}$-representation, $\pi_{G^*} R_n$ contains $\Z_p[\rho^{\pm 1}]$ as a split summand.

\end{cor}

\begin{proof}

We will write $\sigma$ for a a lift of a primitive $p^n-1^{\rm st}$ root of unity in $\F_{p^n}$ to $\W(\F_{p^n})$, so that $\W(\F_{p^n}) = \Z_p[\sigma]$.

Since $R_n$ is the universal multiplicative $n$-oriented spectrum, the multiplicativity of the orientation of $\W(K)$ gives us a unique oriented map $\Phi: R_n \to \W(K)$, which induces $\Phi_{G^*}: \pi_{G^*}(R_n) \to \pi_{G^*}(\W(K))$.   The last Proposition gives us a unique oriented map $\Xi:\pi_{G^*} \W(K) \to \pi_{G^*} R_n [\sigma]$.  It must then be the case that $\Phi_{G^*} \circ \Xi = \id$ once we extend $\Phi_{G^*}$ over $\W(\F_{p^n}) = \Z_p[\sigma]$.  Thus $\pi_{G^*}(R_n)[\sigma]$ contains $\W(\F_{p^n})[\Xi(s)^{\pm1}]$ as a split summand.  As in the proof of Theorem \ref{s_o_thm}, we must have $\Xi(s) = \rho$.  As both maps were equivariant by universality properties, this splitting is equivariant.

We note that $\pi_0(S)$ does not contain any of the roots of unity in $\mu_{p^n-1}$ not lying in $\mu_{p-1}$.  If it contained one such $\tau$, then for any $K(n)$-local ring spectrum $Y$, $\pi_0(Y)$ would be a $\Z_p[\tau]$-algebra.  In particular, $\pi_0(K(n)) = \F_p$ would be an $\F_p[\tau] = \Z_p[\tau] / p$-vector space, which is false if $\tau \notin \mu_{p-1} = \F_p^{\times}$.

Suppose now that $\tau \in \pi_0(R_n)$.  Then the action of $\Z_p^{\times}$ on $\Z_p[\tau] \subseteq \pi_0(R_n)$ is trivial since it is trivial upon extending further to $\Z_p[\sigma] = \W(K)_0$. Examine the long exact sequence
$$\xymatrix{
\cdots \ar[r] & \pi_0(S) \ar[r]^{\eta} & \pi_0(R_n) \ar[r]^-{\psi^g-1} &  \pi_0(R_n) \ar[r] & \cdots
}$$
Then $\Z_p[\tau] \subseteq \ker(\psi^g-1) \subseteq \pi_0(S)$, a contradiction.

Knowing that $\pi_{G^*} R_n [\sigma]$ contains $\W(\F_{p^n})[\rho^{\pm 1}]$ as a split summand, and that $\pi_0 R_n$ does not contain $\sigma$, we see that $\pi_{G^*} R_n$ contains $\Z_p[\rho^{\pm 1}]$ as a split summand.

\end{proof}

\subsection{$R_n$ for large primes}

Let $q:= g^{p-1} = \zeta^{p-1}(1+p)^{p-1} = (1+p)^{p-1}$; $q$ is a topological generator of $(1+p\Z_p)^{\times}$.  Consequently, the equivalence $S = R_n^{h\Z_p^{\times}}$ may be factorised (following \cite{davis_iterated}) as
$$S = (R_n^{h\mu_{p-1}})^{h (1+p\Z_p)^{\times}} = \hofib(\psi^q - 1: R_n^{h\mu_{p-1}} \to R_n^{h\mu_{p-1}})$$

A variation on a standard sparseness result for the Adams-Novikov spectral sequence (see, e.g., \cite{ghmr_rat}) yields the following:

\begin{prop}

If $n^2<2p-3$, the long exact sequence in homotopy associated to 
$$\xymatrix{S \ar[r]^-\eta & R_n^{h\mu_{p-1}} \ar[r]^-{\psi^{q}-1} & R_n^{h\mu_{p-1}}}$$
splits in a range of degrees, giving
$$\begin{array}{ccc}
\Z_p = \pi_0(S) \cong \pi_0(R_n^{h\mu_{p-1}}), & \mbox{and} & \pi_*(S) \cong  \pi_{*}(R_n^{h\mu_{p-1}}) \oplus \pi_{*+1}(R_n^{h\mu_{p-1}})
\end{array}$$
when $-2p+1 \leq * < 0$.

\end{prop}

\begin{proof}

Equivalently, we may show these facts for $E_n^{h\G_n^1} \simeq R_n^{h\mu_{p-1}}$.  The homotopy of this spectrum is computed via the spectral sequence
$$H^s(\G_n^1, (E_n)_t) \implies \pi_{t-s}(E_n^{h\G_n^1}).$$
Since $p$ is odd, our assumption implies that $(p-1)$ and $p$ do not divide $n$, and thus $\G_n$ is a $p$-adic analytic Lie group of dimension $n^2$ with no $p$-torsion; thus, its cohomological dimension is $n^2+1$ (see \cite{morava}).  Similarly, $\G_n^1$ has cohomological dimension $n^2$.  So the only contribution to $\pi_{*}$ comes from $H^s(\G_n^1, (E_n)_{s+*})$ where $0 \leq s \leq n^2$.

However, one may compute this group cohomology as 
$$H^s(\G_n^1 / \mu_{p-1}, (E_n)^{\mu_{p-1}}_{s+*}) = H^s(\SGn, (E_n)^{\mu_{p-1}}_{s+*}),$$
which vanishes when $s+*$ is not a multiple of $2(p-1)$.  Assuming that $n^2<2(p-1)$ and that $-2(p-1)< * \leq 0$ ensures that $-2(p-1) < * \leq s+* \leq s \leq n^2 <2(p-1)$, so the only possible contribution is when $s=-*$ (so $s+* = 0$).  So for $* \leq 0$, 
$$\pi_{*}(R_n^{h\mu_{p-1}}) = H^{-*}(\G_n^1, (E_n)_0)$$
The same analysis holds (for $n^2+1 < 2(p-1)$) to show that when $-2(p-1)< * \leq 0$, $\pi_*(S) =  H^{-*}(\G_n, (E_n)_0)$.  

Now, since $p \nmid n$, the reduced determinant is split, giving an isomorphism $\G_n \cong \G_n^1 \times \Z_p$, and the $\Z_p$ factor acts trivially on $(E_n)_0$ (and hence $H^{*}(\G_n^1, (E_n)_0)$).  Thus the (collapsing) Lyndon-Hochschild-Serre spectral sequence gives $H^0(\G_n, (E_n)_0) \cong H^0(\G_n^1, (E_n)_0)$, and if $*<0$,
\begin{eqnarray*}
H^{-*}(\G_n, (E_n)_0) & \cong & H^1(\Z_p, H^{-1-*}(\G_n^1, (E_n)_0)) \oplus H^0(\Z_p, H^{-*}(\G_n^1, (E_n)_0)) \\
 & \cong & H^{-1-*}(\G_n^1, (E_n)_0) \oplus H^{-*}(\G_n^1, (E_n)_0) \\
\end{eqnarray*}

\end{proof}

This implies that for $0\leq m \leq 2p-1$,
$$1= \psi^{q}: \pi_{-m}R_n^{h\mu_{p-1}} \to \pi_{-m}R_n^{h\mu_{p-1}},$$
since the unit map $S \to R_n^{h\mu_{p-1}}$ is surjective in $\pi_{-m}$, and fixed by $\psi^q$. Now, the $G$-periodicity of $\pi_\bigstar(R_n)$ yields a $G^{\otimes p-1}$-periodicity of $\pi_\bigstar(R_n^{h\mu_{p-1}})$, so 
$$[G^{\otimes k(p-1)}, \Sigma^m R_n^{h\mu_{p-1}}] \cong \pi_{-m}(R_n^{h\mu_{p-1}})$$

\begin{cor} \label{end_cor}

If $n^2<2p-3$, and $0 \leq m\leq 2p-1$, the endomorphism $\psi^q$ of $[G^{\otimes k(p-1)}, \Sigma^m R_n^{h\mu_{p-1}}]$ is multiplication by $q^k$.

\end{cor}

\begin{proof}

It suffices to observe that for $f \in  \pi_{-m}(R_n^{h\mu_{p-1}})$ this commutes:
$$\xymatrix{
\Sigma^{-m} G^{\otimes k(p-1)} = G^{\otimes k(p-1)} \otimes S^{-m} \ar[r]^-{1 \otimes f} \ar[dr]_-{q^k \otimes f} & G^{\otimes k(p-1)} \otimes R_n^{h\mu_{p-1}} \ar[d]^-{g^{k(p-1)} \otimes \psi^q} \ar[r]^-{\delta^{k(p-1)}} & R_n^{h\mu_{p-1}}  \ar[d]^-{\psi^q} \\
 & G^{\otimes k(p-1)} \otimes R_n^{h\mu_{p-1}} \ar[r]_-{\delta^{k(p-1)}} & R_n^{h\mu_{p-1}}
}$$

\end{proof}

\subsection{Analogues of the image of J}

\begin{cor} \label{final_cor}

Let $k \in \Z$, and write $k=p^s m$, where $m$ is coprime to $p$.  Then $[G^{\otimes k(p-1)}, S^1]$ contains a subgroup isomorphic to $\Z / p^{s+1}$.  Furthermore, if $n^2<2p-3$, there is an exact sequence
$$0 \to \Z / p^{s+1} \to [G^{\otimes k(p-1)}, S^1] \to N_{s+1} \to 0$$
where $N_{s+1} \leq \pi_{-1}(S)$ is the subgroup of $p^{s+1}$-torsion elements.

\end{cor}

\begin{proof}

Without assumptions on $n$, Corollary \ref{rep_cor} implies that $[G^{\otimes i}, R_n]$ has a split summand $\Z_p$ upon which the action of $\psi^g$ is through multiplication by $g^i$.  Taking $i = k(p-1)$, we conclude that the action of $\psi^q = \psi^{g^{p-1}}$ on the corresponding summand $\Z_p \leq [G^{\otimes k(p-1)}, R_n^{h\mu_{p-1}}] \cong \pi_0(R_n)$ is by multiplication by $q^k$.

Consider the long exact sequence obtained by applying $[G^{\otimes k(p-1)}, -]$ to the fibre sequence
$$\xymatrix{
\cdots \ar[r] & R_n^{h\mu_{p-1}} \ar[r]^{\psi^q - 1} & R_n^{h\mu_{p-1}} \ar[r] & S^1 \ar[r] & \Sigma R_n^{h\mu_{p-1}} \ar[r]^{\psi^q - 1} & \Sigma R_n^{h\mu_{p-1}} \ar[r] & \cdots.
}$$
The first map contains a factor which is given by multiplication by $q^k-1$ on $\Z_p$.  Since $q^k-1$ generates the procyclic subgroup $p^{s+1} \Z_p \subseteq \Z_p$, we see that $[G^{\otimes k(p-1)}, S^1]$ contains the indicated cokernel as a subgroup.

Now, if $n^2<2p-3$, the results of the previous section give us a very precise computation of this sequence:
$$\xymatrix{
\Z_p \ar[r]^{\psi^q- 1} & \Z_p \ar[r] & [G^{\otimes k(p-1)}, S^1] \ar[r] & [G^{\otimes k(p-1)}, \Sigma R_n^{h\mu_{p-1}}] \ar[r]^-{\psi^q-1} & [G^{\otimes k(p-1)}, \Sigma R_n^{h\mu_{p-1}}],
}$$
and Corollary \ref{end_cor} implies that $\psi^q- 1$ is again multiplication by a unit multiple of $p^{s+1}$.  Lastly, the identification $ [G^{\otimes k(p-1)}, \Sigma R_n^{h\mu_{p-1}}] \cong \pi_{-1}(R_n^{h\mu_{p-1}})$ as a summand of $\pi_{-1}(S)$ (with complement $\Z_p$) yields the desired short exact sequence.

\end{proof}

\section{Redshift} \label{redshift_section}

We will write $M(p)$ for the $K(n)$-local mod $p$ Moore spectrum; i.e., the cofibre of the map $p : S \to S$.  When $p>2$, $M(p)$ is a ring spectrum, homotopy associative if $p>3$; we will assume the latter throughout this section.

\begin{prop} \label{moore_prop}

There is a map $v: \Sdet \to Z \otimes M(p)$ which induces an isomorphism in $E_n$ after multiplying by the identity on $M(p)$:
$$\xymatrix{v_*: {E_n}_*(\Sdet \otimes M(p)) \ar[r]^-\cong & {E_n}_*(Z \otimes M(p))}.$$
Thus $v$ induces an equivalence $\Sdet \otimes M(p) \simeq Z \otimes M(p)$.  

\end{prop}

\begin{proof}

Notice that after tensoring with $M(p)$, the element $g \in \Z_p \subseteq [X_n, X_n]$ becomes homotopic to $\zeta \in \F_p \subseteq [X_n \otimes M(p), X_n \otimes M(p)]$.  Thus the indicated lift in the following diagram exists:
$$\xymatrix{
\Sdet \ar@{.>}[r]^-v \ar[d]_-\delta & Z \otimes M(p) \ar[d]^-{\alpha \otimes 1} \\
R_n \ar[d]_-{\psi^g - g} \ar[r]_-{1 \otimes \eta} & R_n \otimes M(p) \ar[d]^-{(\psi^g - \zeta) \otimes 1} \\
R_n \ar[r]_-{1 \otimes \eta} & R_n \otimes M(p)
}$$
To see that $v_*$ is an isomorphism, we note that the image in 
$${E_n}_*(R_n \otimes M(p)) = C(\Z_p, {E_n}_* /p) = C(\Z_p, \F_{p^n}[[u_1, \dots, u_{n-1}]][u^{\pm 1}])$$
of $\alpha \otimes 1$ and $(1 \otimes \eta) \circ \delta$ are both generated by the function $f_0: \Z_p \to \F_{p^n}[[u_1, \dots, u_{n-1}]][u^{\pm 1}])$ given by $f_0(x) = x \bmod p$.

\end{proof}

This fact allows us to multiply any element of the Picard-graded homotopy of $Y \otimes M(p)$ with $v$, for any spectrum $Y$. 

\begin{prop}

The map $\beta_n^{p-1}: \Sdet^{\otimes p-1} \to L_{K(n)} K(E_{n-1})$ is, after tensoring with $M(p)$, homotopic to multiplication of the unit $\eta:S \to L_{K(n)} K(E_{n-1})$ by $v^{p-1}$.

\end{prop}

\begin{proof}

The statement of the proposition amounts to the claim that the following diagram commutes:
$$\xymatrix{
\Sdet^{\otimes p-1} \ar[r]^-{\rho_n^{p-1}} \ar[d]_-{v^{\otimes p-1}} & X_n \ar[r]^-{i \circ B\varphi_{n-1}} \ar[d]_-{1 \otimes \eta} & L_{K(n)} K(E_{n-1}) \ar[d]^-{1 \otimes \eta} \\
S \otimes M(p) \ar[r]^-{\eta \otimes 1} & X_n \otimes M(p) \ar[r]^-{i \circ B\varphi_{n-1} \otimes 1} & L_{K(n)} K(E_{n-1}) \otimes M(p), \\
}$$
since the lower composite is the unit of $L_{K(n)} K(E_{n-1})$ tensored with the identity on $M(p)$.  The first square commutes by construction, and the second is evident.

\end{proof}

\begin{cor} \label{non_nilp_cor}

Multiplication by $\beta_n \in \pi_{\Sdet} L_{K(n)}K(E_{n-1})$ is an equivalence.

\end{cor}

\begin{proof}

Since $v^{p-1}: \Sdet^{\otimes p-1} \otimes M(p) \to Z^{\otimes p-1} \otimes M(p) = M(p)$ is an equivalence by Proposition \ref{moore_prop}, smashing it with the identity of $L_{K(n)}K(E_{n-1})$ is also an equivalence.  By the previous Proposition, then multiplication by $\beta_n^{p-1}$:
$$[\beta_n^{p-1}] = [v^{p-1}]: (E_n)_{\bigstar} (K(E_{n-1}) \otimes M(p)) \to (E_n)_{\bigstar + \Sdet^{\otimes p-1}} (K(E_{n-1}) \otimes M(p))$$
is an isomorphism.  However, this implies that $\beta_n^{p-1}$ is an isomorphism in $K_*$ and hence an equivalence, since $K_n$ is a module spectrum for $E_n \otimes M(p)$, and hence in the same Bousfield class.

\end{proof}

One may go slightly further.  One of the chromatic redshift conjectures of \cite{ausoni-rognes} (specifically Conjecture 4.4) proposes an equivalence
$$L_{K(n)} K(\Omega_{n-1}) \simeq E_{n},$$
where $\Omega_{n-1}$ is a suitable interpretation of the algebraic closure of the fraction field of $E_{n-1}$.  In particular, this should restrict to a unital $E_\infty$ map $r_n: L_{K(n)} K(E_{n-1}) \to E_{n}$.  When $n=1$, this is a map $L_{K(1)} K(\Q_p) \to KU^{\wedge}_p$ corresponding to a choice of embedding $\Q_p \to \C$.  Let us presume only the existence of such a map $r_n$.  The previous results indicate the commutativity of the solid part of the following diagram:
$$\xymatrix{
L_{K(n)} \Sigma^\infty K(\Z_p, n+1)_+ \ar[d] \ar[r]^-{B \varphi_{n-1}} &  L_{K(n)} \Sigma^\infty B\GL_1(E_{n-1})_+ \ar[r]^-i  & L_{K(n)} K(E_{n-1}) \ar@{.>}[d]^-{r_n} \\
R_{n} \ar[r]^-{\simeq} & E_{n}^{h\SGn} \ar[ur] \ar[r] & E_{n}
}$$
Now, the path along the lower left is the map $\varphi_{n}$.  The path along the upper right (including the conjectural redshift map $r_n$) is an $E_\infty$ map, and thus by Theorem \ref{e_infty_maps_thm}, differs from $\varphi_n$ by at worst an Adams operation.  So up to a suitable re-embedding of $E_{n}^{h\SGn}$ into $E_{n}$, the diagram must commute if $r_n$ is to exist.  This provides a falsifiable condition upon any candidate redshift map $r_n$: it must give a factorisation of the homotopy fixed point inclusion $E_{n}^{h\SGn} \to E_{n}$.

\section{Questions}

This section is purely speculative, and is intended as a collection of vaguely posed questions which the reader is invited to either clarify or (more likely) disprove.

Perhaps the most evident thing lacking in this paper is a geometric definition of the cohomology theory $R_n$.  Despite all of the analogies with K-theory given above, we do not have a description of $R_n^*(X)$ akin to the Grothendieck group of (some generalisation of) vector bundles over $X$.  The discussion of orientations on $p$-adic $(n-1)$-gerbes indicates a possible direction with this question, but it is not apparent to the author what the corresponding analogue of higher rank vector bundles should be (although the work of Michael Murray and others on bundle gerbe modules \cite{murray, bcmms} may point a way forward).

A related deficiency is the fact that we have described a $K(n)$-local analogue of the \emph{image} of the J-homomorphism, but not precisely the homomorphism itself.  We do have a monochromatic formulation of the homomorphism, but it would be ideal to extend this to a geometric construction defined on $R_n$, not the Moore-Postnikov factorisation $Q$.  Classically, one can think of the J-homomorphism as built from either the functions $O(m) \to \Omega^m S^m$, or (in families, after delooping) the function that assigns to a vector bundle the associated sphere bundle.  Neither of these has an obvious analogue in our construction.  Such a description would be enlightening.  In connection with the first question this raises the obvious question: does an $n$-bundle gerbe have a geometrically defined ``sphere" bundle, where $\Sdet$ is our replacement for the spherical fibre?  

This, in turn, raises the question of whether there is a good unstable description of the map $\rho_n: \Sdet \to L_{K(n)} \Sigma^\infty K(\Z_p, n+1)_+$ analogous to the inclusion $\C P^1 \to \C P^\infty$.  It follows from \cite{bousfield_telescope} that $\Sdet$ is the $K(n)$-localisation of a suspension spectrum, but a more geometric description of this spectrum and map is desirable.

Further, the definition of an $n$-oriented spectrum raises the question of whether it is possible to repeat the whole program of analysing complex-oriented spectra by their formal group laws, but in the $K(n)$-local category, with $K(\Z_p, n+1)$ replacing $\C P^\infty$.  All of our examples, however, have resulted in essentially multiplicative formal groups.  One standout is the spectrum $MX_n$ for which we haven't even the beginnings of a computation of the associated formal group when $n>1$.  We would like to believe that it plays a universal role analogous to $MU$ in the complex-oriented case, but do not have any evidence to back this up.  

Lastly, we have no examples in hand of $n$-oriented spectra whose associated formal group law is \emph{additive} (i.e., the analogue of singular homology).  Does such a theory exist?  If so, and if it could be made part of an Atiyah-Hirzebruch spectral sequence for $K(n)$-local theories, the rather convoluted construction of the formal group law in Theorem \ref{orient_thm} (requiring the entirety of section 4) could be replaced with a direct analogue of the argument familiar to most homotopy theorists for the proof that $h^*(\C P^\infty) \cong h_*[[x]]$ for complex-oriented theories $h$.  However, when $n=1$, this argument requires singular cohomology, a non-$K(1)$-local theory.  So to mimic this argument, our definition of $n$-oriented spectrum would need to be extended beyond the $K(n)$-local setting to a larger class of spectra.

\bibliography{biblio}

\end{document}